\documentclass[12pt,a4paper,oneside,reqno]{amsart}
\usepackage[english]{babel}
\usepackage{amsmath}
\usepackage{amsfonts}
\usepackage{amsthm}
\usepackage{amssymb}
\usepackage[T1]{fontenc}
\usepackage{ae}
\usepackage{epsfig}
\usepackage{graphicx}
\usepackage{times}
\usepackage{latexsym}
\usepackage{xspace}
\usepackage{color}
\usepackage{pstricks}
\usepackage{pstcol}
\usepackage[colorlinks=true,linkcolor=red,citecolor=blue]{hyperref}
\usepackage{accents}
\theoremstyle{plain}

\newtheorem{teo}{Theorem}[section]

\newtheorem{pro}[teo]{Proposition}
\newtheorem{lem}[teo]{Lemma}
\newtheorem{cor}[teo]{Corollary}

\newtheorem*{teo13}{Theorem \ref{teo:main3}}

\theoremstyle{definition}
\newtheorem{defi}[teo]{Definition}

\numberwithin{equation}{teo}

\theoremstyle{remark}
\newtheorem{rem}[teo]{Remark}
\newtheorem*{example*}{Example}

\begin{document}

\title[Nilpotent actions on the torus]{Global fixed points for\\ nilpotent actions on the torus}

\author{S. Firmo and J. Rib\'on}
\address{Instituto de Matem\'{a}tica e Estat\'\i stica \\
Universidade Federal Fluminense\\
Campus do Gragoat\'a\\
Rua Marcos Valdemar de Freitas Reis s/n, 24210\,-\,201 Niter\'{o}i, Rio de Janeiro - Brasil }
\email{firmo@mat.uff.br}
\email{jribon@id.uff.br}
\subjclass{Primary:  37E30, 37E45, 37A15, 37A05 \ ; \ Secondary: 55M20, 37C25}

\thanks{Keywords: rotation vector, global fixed point, derived group, homeomorphism, diffeomorphism,
irrotational homeomorphism, ergodic probability measure,  nilpotent group}
\date{\today}
\thanks{Supported in part by CAPES}

\thispagestyle{empty}

\begin{abstract}

An isotopic to the identity map of the  \,$2$-torus, that has zero
rotation vector with respect
to an invariant ergodic probability measure, has a fixed point by a theorem of Franks.
We give a version of this result
for nilpotent subgroups of isotopic to the identity diffeomorphisms of the \,$2$-torus.
In such a context we guarantee the existence of global fixed points for
nilpotent groups of
irrotational diffeomorphisms. In particular we show that the derived group of a nilpotent group
of isotopic to the identity diffeomorphisms of the $2$-torus has a global fixed point.

\end{abstract}

\maketitle

\thispagestyle{empty}

\vskip20pt
\section{Introduction}

 Abelian groups of isotopic to the identity  \,$C^{1}$-diffeomorphisms
of a closed oriented surface \,$S$\, have global fixed points when the Euler characteristic
of \,$S$\, is different from zero. When \,$S$\, is the \,$2$-sphere it is also assumed  that
the group is generated by diffeomorphisms \,$C^{0}$-close to the identity map of \,$S$.

In these cases the existence of global fixed points is essentially imposed by the topology of \,$S$\,
and its interplay with the dynamics of nilpotent (or abelian)  actions.
 Several papers have dealt with this question
as \cite{bo01,bo02,han01,fir01,fhp01,fhp02,fir03} on the abelian case, and as
\cite{fir02,JR:arxivsp,rib01} for the nilpotent ones.

The starting point for this kind of problem has been the articles of Lima \cite[1964]{elon02} and Plante  \cite[1986]{pl01} where   they proved the existence of global fixed points for the action of abelian and nilpotent connected Lie groups respectively, on surfaces with Euler characteristic different from zero.

When \,$S$\, is the \,$2$-torus \,$\mathbb{T}^{2}$\, we need to assume more conditions on the group,
other than the nilpotent property, to guarantee the existence of global fixed points.

In a recent paper, \,{\it Finite orbits for nilpotent actions on the torus}\, (cf. \cite{rib02}),
we study the existence of finite orbits for the action
assuming   hypothesis of topological type on the group. More precisely we show that if a nilpotent group
of $C^{1}$ diffeomorphisms of the $2$-torus has an element with non-vanishing Lefschetz number then it
has a finite orbit.
In this article we follow a different approach, considering hypothesis of dynamical nature
on the group that are inspired in a result of Franks (cf. Theorem \ref{teo:franks}).

 To approach this question let
us consider a homeomorphism \,$\phi$\, of the torus \,$\mathbb{T}^{2}$\, that belongs to
the connected component of the identity, \,$\mathrm{Homeo}_{0}({\mathbb T}^{2})$\,,
of \,$\mathrm{Homeo}({\mathbb T}^{2})$. Let \,$\tilde{\phi}$\, be a lift of \,$\phi$\, to
the universal covering \,${\mathbb R}^{2}$\, of \,${\mathbb T}^{2}$.
Following \cite{mi03}, the rotation vectors
of \,$\tilde{\phi}$\, are the limits of  sequences of the form
\[ \frac{\tilde{\phi}^{\,n_{k}}(\tilde x_{k}) - \tilde x_{k}}{n_{k}} \]
where \,$(n_{k})_{k \geq 1}$\, is  an increasing sequence  of positive integers and
\,$(\tilde x_{k})_{k \geq 1}$\, is a sequence of points in \,$\mathbb{R}^{2}$.
This set will be denoted by \,$\rho\big(\tilde\phi\big)$. It follows from
\cite{mi03} that it is a non-empty compact and convex subset of \,$\mathbb{R}^{2}$.

Equivalently we have that
\,$\rho \big(\tilde{\phi}\big) = \big\{ \rho_{\mu} \big(\tilde{\phi}\big) \ \, ; \ \mu \in {\mathcal P}(\phi) \big\}$\,
where \,${\mathcal P}(\phi)$\, is the set of $\phi$-invariant Borel probability measures and
$$ {\rho}_{\mu}\big(\tilde{\phi}\big): = \int_{{\mathbb T}^{2}} \big(\tilde{\phi} - Id \big)
 \ d \mu \,. $$
Since \,$\phi$\, belongs to \,$\mathrm{Homeo}_{0}({\mathbb T}^{2})$,
it follows that
\,$\tilde{\phi}$\, commutes with the covering transformations and
\,$\tilde{\phi}-Id$\, is a well-defined map on \,${\mathbb T}^{2}$.
The set \,$\rho \big(\tilde{\phi}\big)$\,  depends
on the lift \,$\tilde{\phi}$\, of \,$\phi$\, and satisfies
\,$\rho_{\mu} \big(T_{v} \circ \tilde{\phi}\big) = T_{v}\big(\rho_{\mu} \big(\tilde{\phi}\big)\big)$\,
for any \,$\mu \in {\mathcal P}(\phi)$\, where
 \,$T_{v}$\, is the  translation in \,$\mathbb{R}^{2}$\, by the vector
 \,$v\in\mathbb{Z}^{2}$.
In particular the projection \,$\rho (\phi)$\, of \,$\rho \big(\tilde{\phi}\big)$\, in
\,${\mathbb R}^{2}/{\mathbb Z}^{2}$\,  does not depend on the lift
\,$\tilde{\phi}$\, of \,$\phi$.

\begin{teo}[Franks \cite{f22}]
\label{teo:franks}
Let \,$\phi \in \mathrm{Homeo}_{0}({\mathbb T}^{2})$. Let \,$\mu$\, be an ergodic $\phi$-invariant Borel probability
measure such that the rotation vector \,$\rho_{\mu}(\phi)$\, is  the zero element
of \,${\mathbb R}^{2}/{\mathbb Z}^{2}$. Then \,$\phi$\, has a fixed point.
\end{teo}

The next result is a generalization
of this theorem in the context of nilpotent groups of diffeomorphisms.

\vglue5pt

\noindent
 {\it Conventions}. From now on, we make the following conventions. A homeomorphism
\,$\tilde{\psi}\in \text{Homeo}(\mathbb{R}^{2})$\, always denotes a lift to the universal covering of
\,${\psi}\in \text{Homeo}(\mathbb{T}^{2})$\, and vice-versa. Moreover,
\,$\pi:\mathbb{R}^{2}\rightarrow\mathbb{T}^{2}$\, denotes the universal covering map and
unless explicitly stated otherwise a lift means a lift to the universal covering.

\vskip10pt
\begin{teo}
\label{teo:main4}
Let \,$G= \langle H,\phi \rangle$\, be a nilpotent subgroup of
\,$\mathrm{Diff}_{0}^{1}({\mathbb T}^{2})$\,
where \,$H$\, is a normal subgroup of \,$G$. Suppose  there exists a $\phi$-invariant
ergodic probability measure \,$\mu$\, such that the support of \,$\mu$\,
is contained in
\,$\mathrm{Fix}(H)$\, and \,$\rho_{\mu}(\phi)$\, is the zero element in \,$\mathbb{R}^{2}/\mathbb{Z}^{2}$.
Then \,$G$\, has a global fixed point.
\end{teo}

We denote by \,$\langle H,\phi \rangle$\,  the group generated by \,$H$\, and \,$\phi$.
Moreover, \,$\mathrm{Fix}(\phi)$\, denotes the set of fixed points
of  \,$\phi$.
We define  \,$\mathrm{Fix}(L) = \cap_{f \in L} \mathrm{Fix}(f)$\,
as the set of global fixed points of a set \,$L$\, of transformations.
By definition \,$\mathrm{Diff}_{0}^{1}({\mathbb T}^{2})$\, is the group of
$C^{1}$-diffeomorphisms of \,$\mathbb{T}^{2}$\, which are \,$C^{0}$-isotopic
to the identity.

Let us remark that the fixed point detected by the theorem of Franks is not
necessarily contained in the support of the invariant ergodic measure.
However, if the homeomorphism has a unique fixed point then it
belongs to the aforementioned support.
This follows from the next two remarks. If \,$\phi$\, has a unique fixed point
\,$p\in\mathbb{T}^{2}$\, then the fixed point index of \,$\phi$\, at \,$p$\, is zero
by the Lefschetz fixed-point theorem.
In this case, a result of
Schmitt in \cite{sch02} guarantees that we can
perturb \,$\phi$\, in an arbitrarily small neighborhood of \,$p$\,
to obtain \,$\psi  \in \textrm{Homeo}_{0}(\mathbb{T}^{2})$\, with no fixed points.

Of course, from the remarks above is simple to conclude the existence of a global fixed point for \,$\langle H,\phi \rangle$\, in Theorem \ref{teo:main4} when \,$\phi$\, has a unique fixed point even in the case where \,$\langle H,\phi \rangle \subset
\text{Homeo}_{0}(\mathbb{T}^{2})$\,.

The previous theorem guarantees, in particular, that two commuting elements
\,$\phi$\, and \,$\psi$\, of
\,$\mathrm{Diff}_{0}^{1}({\mathbb T}^{2})$\, share a common fixed point
if there exists a $\phi$-invariant ergodic  probability measure \,$\mu$\, whose support
is contained in \,$\mathrm{Fix}(\psi)$\, and
\,$\rho_{\mu}(\phi)=0\in\mathbb{R}^{2}/\mathbb{Z}^{2}$.
Notice that the original Franks theorem
holds true for homeomorphisms but we need the differentiability since we use
a Thurston decomposition \,{\it \`{a} la}\, Franks-Handel-Parwani \cite{fhp01}.



We say that an element \,$\phi$\, of
$\mathrm{Homeo}_{0}({\mathbb T}^{2})$\, is
\,{\it irrotational}\,  if there exists a  lift
\,$\tilde{\phi}\in \mathrm{Homeo}_{0}({\mathbb R}^{2})$\, of \,$\phi$\,  that
has \,$(0\,,0)$\, as its unique rotation vector,
i.e. \,$\rho_{\mu}\big(\tilde{\phi}\big)=(0\,,0)$\, for any  \,$\mu\in\mathcal{P}(\phi)$.
In that  case \,$\tilde{\phi}$\, is unique and  we say that it
is the \,{\it irrotational lift}\, of \,$\phi$.
A subgroup \,$G$\, of \,$\mathrm{Homeo}_{0}({\mathbb T}^{2})$\, is irrotational if
every of its elements is irrotational.

\vglue5pt


 The result of Franks \cite{f22} implies, in particular, that the irrotational elements of the group
\,$\text{Homeo}_{0}(\mathbb{T}^{2})$\, have fixed points. Following the context of Franks and of
Theorem \ref{teo:main4},
it is natural to study the existence of global fixed points for
irrotational  nilpotent  subgroups of \,$\mathrm{Diff}_{0}^{1}({\mathbb T}^{2})$.
We solve this problem in the next theorem.
It  will follow as a consequence of Theorem \ref{teo:main4}
and  simple well-known properties of nilpotent groups.

\vskip10pt
\begin{teo}
\label{teo:main2}
Let \,$G$\, be an irrotational nilpotent subgroup of \,$\mathrm{Diff}_{0}^{1}({\mathbb T}^{2})$.
Then \,$G$\, has a global fixed point.
\end{teo}

Let us remark that the existence of fixed points for each element of a nilpotent subgroup \,$G$\, of \,$\mathrm{Diff}^{1}_{0}(\mathbb{T}^{2})$\, is not sufficient to guarantee a global fixed point for \,$G$. In the above result we exploit the fixed points given by the irrotational nature of the elements of \,$G$.

The dynamics of irrotational homeomorphisms
has been studied by Koropecki and Tal (cf. \cite[Theorem A]{koro01})
in the area-preserving case.



 Motivated by Theorem \ref{teo:main2}
it is natural to search for examples of irrotational groups.
The next theorem illustrates that
irrotational groups appear naturally in the context of nilpotent groups.

%

\vskip10pt
\begin{teo}
\label{teo:main3}
If \,$G$\, is a nilpotent subgroup of \,$\mathrm{Homeo}({\mathbb T}^{2})$ then$\,:$
\begin{itemize}
\item[$(i)$]
$G':=[\,G,G\,]$\, is irrotational when \,$G \subset \mathrm{Homeo}_{0}({\mathbb T}^{2})\,;$

\item[$(ii)$]
$G'':=[\,G',G'\,]$\,  is irrotational when  \,$G \subset \mathrm{Homeo}_{+}({\mathbb T}^{2})\,;$
\item[$(iii)$]
$G''':=[\,G'',G''\,]$\, is always irrotational.

\end{itemize}

\end{teo}

The group \,$\mathrm{Homeo}_{+}({\mathbb T}^{2})$\, is the subgroup of
\,$\mathrm{Homeo}({\mathbb T}^{2})$\, formed by the orientation-preserving elements.
Theorems \ref{teo:main3} and \ref{teo:main2}   provide immediately the following result.

%

\vskip10pt
\begin{teo}
\label{teo:main}
Let \,$G$\, be a nilpotent subgroup of \,$\mathrm{Diff}^{1}({\mathbb T}^{2})$. The following subgroups of \,$G$\, have a global fixed point$\,:$
\begin{itemize}
\item[$(i)$]
$G'$\, when \,$G \subset \mathrm{Diff}^{1}_{0}({\mathbb T}^{2})\,;$

\item[$(ii)$]
$G''$\, when \,$G \subset \mathrm{Diff}^{1}_{+}({\mathbb T}^{2})\,;$
\item[$(iii)$]
$G'''$.

\end{itemize}

\end{teo}

Given a subgroup \,$G \subset \mathrm{Homeo}({\mathbb T}^{2})$\, we always denote by \,$G_{0}$\,
the subgroup  of isotopic to the identity elements of \,$G$\, and we
define the sets\,:
\[ {\mathcal I}_{G} := \big\{ \phi \in G_{0} \ \ ; \ \exists \ \ \text{a lift} \
\ \tilde{\phi} \ \ \text{s.t.} \ \ \rho_{\mu}(\tilde{\phi})=(0\,,0) \ \
\text{for all} \ \  \mu \in {\mathcal P} (\phi) \big\} \]
and
\[ G_{\mathcal I} := \big\{ \phi \in G_{0} \ \ ; \ \ \exists \ \ \text{a lift} \ \
\tilde{\phi} \ \ \text{s.t.} \ \ \rho_{\mu}(\tilde{\phi})=(0\,,0) \ \
\text{for all} \ \ \mu \in {\mathcal P} (G) \big\} \]
where \,${\mathcal P}(G) = \cap_{\psi \in G} \, {\mathcal P}(\psi)$\,.
Of course we are following our convention that \,$\tilde{\phi}$\, is a lift of
\,$\phi \in \text{Homeo}(\mathbb{T}^{2})$\, to the universal covering.
The set \,${\mathcal I}_{G}$\, consists of the irrotational elements of \,$G$\,
whereas \,$G_{\mathcal I}$\, is the set of elements of \,$G$\, that are irrotational with
respect to \,$G$-invariant measures. It is clear that \,${\mathcal I}_{G}$\,
is contained in \,$G_{\mathcal I}$.
Moreover, \,$G_{\mathcal{I}}$\, is a normal subgroup of
\,$G$\, by Lemmas \ref{rem:pact} and \ref{lem:mor}.

\vglue10pt

Given a surface \,$S$\, we fix the notations \,$\mathrm{Homeo}(S)$\,,
\,$\mathrm{Homeo}_{+}(S)$\, and \,$\mathrm{Homeo}_{0}(S)$\, with similar
meanings as we have done previously,
for the case \,$S=\mathbb{T}^{2}$\, and \,$S=\mathbb{R}^{2}$.
We do the same with the notations \,$\mathrm{Diff}^{1}(S)$\,,
\,$\mathrm{Diff}^{1}_{+}(S)$\, and \,$\mathrm{Diff}^{1}_{0}(S)$\, for sets of \,$C^{1}$-difeomorphisms
of \,$S$.

Let \,$G$\, be a subgroup of \,$\mathrm{Homeo}(S)$\, and let \,$\tau:\tilde{S}\rightarrow S$\, a covering map.  We say that
a subgroup \,$G_{\tau}$\, of  \,$\mathrm{Homeo}(\tilde{S})$\,
is a \,{\it \,$\tau$-lift}\, of \,$G$\,  if each element \,$\phi_{\tau} \in G_{\tau}$\,
is a \,$\tau$-lift of  some element \,$\phi \in G$\, and
the  natural projection \,$\kappa: G_{\tau} \to G$\, defined by \,$\kappa (\phi_{\tau})=\phi$\, is an
isomorphism of groups.

\vskip10pt

In the study of
nilpotent groups of homeomorphisms  of the $2$-torus, the
irrotational  elements  play a similar role   as the isotopic to the
identity   homeomorphisms    in closed orientable surfaces
\,$S$\, with negative Euler characteristic.
In both cases there exists a privileged lift, namely the irrotational lift
and the \,{\it identity lift}\,,  respectively, for irrotational and isotopic to the identity
homeomorphisms.

Let us remind the reader that
the negative Euler characteristic hypothesis on \,$S$\, guarantees that
\,$\mathrm{Homeo}_{0}(S)$\, is homotopically trivial (cf. \cite{meh01})
and we have that the \,{\it identity lift}\,
\,$\tilde{\phi}$\, of \,$\phi \in \mathrm{Homeo}_{0}(S)$\,
is the unique lift
of \,$\phi$\, to the universal covering of \,$S$\, that commutes with all covering
transformations. It is obtained by lifting, to the universal covering
\,$\tilde{S}$\, of \,$S$\,, any isotopy from \,${Id}_{S}$\, to \,$\phi$\, starting at
\,${Id}_{\tilde{S}}$\,.  The set \,$\tilde{\mathcal{I}}_{\text{Homeo}_{0}(S)}$\, of identity lifts of all elements in \,$\mathrm{Homeo}_{0}(S)$\,
is a normal subgroup of the group of all lifts of elements of \,$\mathrm{Homeo}(S)$.
Analogously the set
\,$\tilde{\mathcal I}_{\mathrm{Homeo}_{0}({\mathbb T}^{2})}$\,, of irrotational lifts
of elements of \,${\mathcal I}_{\mathrm{Homeo}_{0}({\mathbb T}^{2})}$\,,
is not a group but it is a \,{\it normal subset}\, of the group of all lifts of elements of
\,$\mathrm{Homeo}(\mathbb{T}^{2})$\, in the following sense\,:
\,$\tilde{\psi} \circ \tilde{\phi} \circ \tilde{\psi}^{-1} \in
\tilde{\mathcal I}_{\mathrm{Homeo}_{0}({\mathbb T}^{2})}$\,
if \,$\tilde{\phi} \in \tilde{\mathcal I}_{\mathrm{Homeo}_{0}({\mathbb T}^{2})}$\, and \,$\tilde{\psi}$\, is a lift
of any \,$\psi \in \mathrm{Homeo}({\mathbb T}^{2})$\,
to the universal covering of \,$\mathbb{T}^{2}$. This property is
a consequence of Lemma \ref{rem:pact}. Similarly, the set \,${\mathcal I}_{\mathrm{Homeo}_{0}({\mathbb T}^{2})}$\, has the same property.

Anyway we prove that even if \,${\mathcal I}_{G}$\, is not a group for a
general subgroup \,$G$\, of the group \,$\mathrm{Homeo}({\mathbb T}^{2})$, it is always
a normal subgroup of \,$G$\, when \,$G$\, is a nilpotent
subgroup of \,$\mathrm{Homeo}(\mathbb{T}^{2})$.
A natural strategy in order to show this result is
to compare \,${\mathcal I}_{G}$\, with \,$G_{\mathcal I}$.
More precisely, we compare
the set of rotation vectors of a
lift \,$\tilde{\phi}$\, of \,$\phi \in G_{0}$\, with the set
\,$\{ \rho_{\mu}(\tilde\phi) \ \, ; \  \mu \in {\mathcal P}(G)\}$\, of rotation vectors
with respect to $G$-invariant measures.
In such context we prove the following result.

\vskip10pt
\begin{teo}
\label{teo:main6}
Let \,$G$\, be a nilpotent subgroup
of \,$\mathrm{Homeo}({\mathbb T}^{2})$.
Then there exists a finite index normal subgroup \,$H$\, of
\,$G$\, containing \,$G_{0}$\, such that
\[
\big\{ \rho_{\mu}\big(\tilde{\phi}\big) \ \ ; \  \mu \in {\mathcal P}(\phi)\big\} =
\big\{ \rho_{\nu}\big(\tilde{\phi}\big) \ \ ; \  \nu \in
  {\mathcal P}(H)\big\}   \]
for any lift \,$\tilde{\phi}$\, of every  \,$\phi \in G_{0}$\,.
In particular, the set \,${\mathcal I}_{G}$\,
of irrotational elements of \,$G$\, is a normal subgroup of \,$G$.
Moreover, if \,$G=G_{0}$\,  then
\,${\mathcal I}_{G}=G_{\mathcal I}$.
\end{teo}

It is possible to choose the group $H$ such that
$G/H$ is equal to a subgroup of $C_{6}$ or a subgroup of $D_{4}$
where $C_{n}$ is the cyclic group with $n$ elements and $D_{m}$ is the
dihedral group with $2m$ elements.

 From Theorems \ref{teo:main6}\,, \ref{teo:main2} and \ref{teo:main3} we easily conclude the following corollary.

\vskip10pt
\begin{cor}
 Let \,$G$\, be a nilpotent subgroup of \,$\mathrm{Diff}^{1}_{0}(\mathbb{T}^{2})$.
The set of irrotational elements of \,$G$\, is a normal subgroup of \,$G$\,
that contains \,$G'$\, and has a global fixed point.
\end{cor}

Other results comparing the rotational properties of the elements of a
nilpotent subgroup of $\mathrm{Homeo}({\mathbb T}^{2})$ with the rotational
properties with respect to measures in ${\mathcal P}(G)$ are
introduced in section \ref{subsec:consinv} (cf. Propositions \ref{pro:bigsma}, \ref{pro:bigsma2} and
\ref{pro:bigsma3}).

\vskip10pt

\vskip30pt
\section{Preliminaries}
\label{sec:pre}

Let \,$G$\, be a group and \,$H$\, be a subgroup of \,$G$. We denote by \,$[\,H,G\,]$\, the subgroup of
\,$G$\, generated by the elements of the form \,$[\,h\,,g\,]=hgh^{-1}g^{-1}$\, where
\,$h\in H$\, and \,$g\in G$. If \,$H$\, is a normal subgroup of \,$G$\, then \,$[\,H,G\,]$\,
is a  subgroup of \,$H$\, which is normal in \,$G$.

Given a group \,$G$\,  let us consider  the \,{\it upper central series}\,
\,$\{Z^{(n)}(G)\}_{n\geq0}$\, of \,$G$\,
$$Z^{(n+1)}(G):= \big\{g\in G \ \,; \ [\,g\,,f\,]\in Z^{(n)}(G) \ \ \text{for all} \ \ f\in G \big\}$$
where \,$Z^{(0)}(G)$\, is the trivial subgroup of \,$G$.
The members of the upper central series are characteristic subgroups of \,$G$. In particular they are normal
subgroups of \,$G$\, and we have
$$Z^{(0)}(G)\subset Z^{(1)}(G)\subset \cdots \subset Z^{(n)}(G) \subset \cdots \subset G .$$
If \,$Z^{(n)}(G)=G$\, for some \,$n\in \mathbb{Z}_{\geq0}$\, we say that \,$G$\, is a
\,{\it nilpotent group}. The smallest \,$n \in \mathbb{Z}_{\geq0}$\,
 such that \,$Z^{(n)}(G)=G$\, is the \,{\it nilpotency class}\, of \,$G$.

\vskip5pt

 The focus of this paper is the study of nilpotent groups of
homeomorphisms and diffeomorphisms of the $2$-torus.




\vskip10pt
\begin{defi}
\label{def:imcg}
Let \,$[\,]: \mathrm{Homeo}({\mathbb T}^{2}) \to \mathrm{MCG}({\mathbb T}^{2})=
\mathrm{GL}(2\,,{\mathbb Z})$\,
be the map associating to an element
of \,$\mathrm{Homeo}({\mathbb T}^{2})$\, its image in the mapping class group  of
\,$\mathbb{T}^{2}$.
Given a subgroup \,$G$\, of \,$\mathrm{Homeo}({\mathbb T}^{2})$\, we denote by \,$[G]$\,
the image of \,$G$\, by $[\,]$.
\end{defi}
\vskip5pt


In what follows we will be using frequently the next two  well known lemmas.

\begin{lem}
\label{rem:pact}
Let \,$\phi \in \mathrm{Homeo}_{0}({\mathbb T}^{2})$\,, \,$\mu\in\mathcal{P}(\phi)$\, and
\,$\psi \in \mathrm{Homeo}({\mathbb T}^{2})$\,.
For all lifts to the universal convering \,$\tilde{\phi}$\, and \,$\tilde{\psi}$\,  of \,$\phi$\, and \,$\psi$\,  respectively,
we have$\,:$
\[\tilde{\psi} \circ T_{v}= T_{[\psi](v)} \circ \tilde{\psi} \quad \text{and} \quad
[\psi]\Big(\rho_{\mu}\big(\tilde{\phi}\big)\Big) =
\rho_{\nu}\big(\tilde{\psi} \circ \tilde{\phi} \circ \tilde{\psi}^{-1}\big)  \]
where \,$v\in\mathbb{Z}^{2}$\, and
 \,$\nu=\psi_{*}(\mu)$\,.
\end{lem}

\vskip5pt

\begin{lem}
\label{lem:mor}
Consider the subgroup \,$\mathrm{Homeo}_{0,\mu}({\mathbb T}^{2})$\,
of \,$\mathrm{Homeo}_{0}({\mathbb T}^{2})$\, whose elements preserve a
probability measure \,$\mu$\, and let \,$\mathrm{Homeo}_{0,\mu}({\mathbb R}^{2})$\,
be the subgroup of \,$\mathrm{Homeo}_{0}({\mathbb R}^{2})$\,  consisting of all the lifts of elements of \,$\mathrm{Homeo}_{0,\mu}({\mathbb T}^{2})$.
 Then the map
$$\rho_{\mu}: \mathrm{Homeo}_{0,\mu}({\mathbb R}^{2}) \to {\mathbb R}^{2}$$
is a morphism of groups.
\end{lem}

\vskip5pt

Now we introduce three results proved by the second author
in \cite{JR:arxivsp}
that will be used to find global fixed points
for nilpotent groups of  torus diffeomorphisms.

\vskip10pt
\begin{teo}
\label{teo:plane}
Let \,$G \subset \mathrm{Diff}_{+}^{1}({\mathbb R}^{2})$\, be a finitely generated
nilpotent subgroup
that preserves a non-empty compact set. Then \,$G$\, has a global fixed point.
\end{teo}

\begin{teo}
\label{cor:plane}
Let \,$G \subset \mathrm{Diff}_{+}^{1}({\mathbb R}^{2})$\, be a
nilpotent subgroup such that
\,$\mathrm{Fix}(\phi)$\, is a non-empty compact set for some \,$\phi \in G$.
Then \,$G$\, has a global fixed point.
\end{teo}
\vskip5pt


The next result is a Thurston decomposition for a normal subgroup of
a nilpotent group. It is inspired in similar results given in
\cite{fhp01} for the abelian case.


\begin{pro}
\label{pro:tnf}
Let \,$G$\, be a finitely generated nilpotent subgroup of \,$\mathrm{Diff}_{+}(S)$\, where \,$S$\,
is a finitely punctured surface. Consider a normal subgroup \,$H$\, of \,$G$\, and a compact $G$-invariant
set \,$K$\, such that \,$K \subset \mathrm{Fix}(H)$. Denote \,$M =S \setminus K$\, and suppose that
the Euler characteristic \,$\chi (M)$\, is negative if \,$K$\, is finite.
Then there exist a finite family \,${\mathcal R}$\, of pairwise disjoint, pairwise non-homotopic simple closed curves
in \,$M$\, and a family of pairwise disjoint open annuli \,${\mathcal A}$\, in \,$M$\, of the same cardinal as
the family \,${\mathcal R}$\, such that
every curve in \,${\mathcal R}$\, is a core curve of exactly an annulus in \,${\mathcal A}$\, and given \,$\phi \in G$\,
there exists \,$\theta_{\phi} \in \mathrm{Homeo}_{+}(S)$\, such that \,$\phi$\, is isotopic to \,$\theta_{\phi}$\,
relative to \,$K$\, and \,$\theta_{\phi} ({\mathcal A})= {\mathcal A}$. Moreover given a connected component \,$X$\,
of \,$S \setminus {\mathcal A}$\, we have
\begin{itemize}
\item $\theta_{\phi} \equiv Id$\, if \,$\sharp (X \cap K) = \infty$\, and \,$\phi \in H \,;$
\item $\chi (X \setminus K) < 0$\, if \,$X \cap K$\, is finite. Moreover the set
\,$\{ (\theta_{\phi})_{|X \setminus K}: \phi \in J\}$, where \,$J \subset G$\, is the stabilizer of \,$X$\, modulo
isotopy relative to \,$K$, is a group of pseudo-Anosov and finite order elements that is virtually cyclic and
nilpotent.
\end{itemize}
\end{pro}


The result is obtained by considering
the subgroups \,$H_{1}=H$\,, \,$H_{2}=G$\, and the sets
\,$K_{1} = K$\,
and \,$K_{2}=\emptyset$\, in Proposition 3.1 of \cite{JR:arxivsp}.

\vskip30pt
\section{Rotational properties}
\label{sec:rot}

In the introduction we say that a subgroup \,$G$\, of \,$\text{Homeo}_{0}(\mathbb{T}^{2})$\,
is irrotational when every element  \,$\phi \in G$\, is irrotational with respect to
\,$\mathcal{P}(\phi)$\,, i.e. there exists a lift
\,$\tilde{\phi}\in \text{Homeo}_{0}(\mathbb{R}^{2})$\, of \,$\phi$\, such that
\,$\rho_{\mu}\big(\tilde{\phi}\big)=(\,0\,,0\,)$\, for all
\,$\mu \in \mathcal{P}(\phi)$. Following the definition of \,$\mathcal{I}_{G}$\, we have that a subgroup
\,$G \subset \text{Homeo}_{0}(\mathbb{T}^{2})$\,
is irrotational when \,$G=\mathcal{I}_{G}$\,.

 Now, let us introduce a new definition of  irrotational subgroups of
\,$\text{Homeo}_{0}(\mathbb{T}^{2})$.
We say that a subgroup  \,$G \subset \text{Homeo}_{0}(\mathbb{T}^{2})$\, is
\,{\it $\mathcal{P}(G)$-irrotational}\,
when every element \,$\phi \in G$\, is irrotational with respect to
\,$\mathcal{P}(G)$\,, i.e. there exists a lift
\,$\tilde{\phi}\in \text{Homeo}_{0}(\mathbb{R}^{2})$\, of \,$\phi$\, such that
\,$\rho_{\mu}\big(\tilde{\phi}\big)=(\,0\,,0\,)$\, for all
\,$\mu \in \mathcal{P}(G)$.
Following the definition of \,$G_{\mathcal{I}}$\, we have that a subgroup
\,$G \subset \text{Homeo}_{0}(\mathbb{T}^{2})$\,
is $\mathcal{P}(G)$-irrotational when \,$G=G_{\mathcal{I}}$\,.

These two concepts of irrotationality could be a priori different since given
\,$\phi \in G$\, the set
\,${\mathcal P}(\phi)$\,  can be much bigger than
\,${\mathcal P} (G)$.
Later on we will see that if \,$G$\, is a nilpotent subgroup of
\,$\mathrm{Homeo}_{0}({\mathbb T}^{2})$\, the definitions coincide
(cf. Proposition \ref{pro:bigsma2}).

Now we introduce some notations.
Let \,$G$\, be a subgroup of \,$\text{Homeo}(\mathbb{T}^{2})$\, and suppose \,$\mu\in\mathcal{P}(G)$\,. Following the convention about \,$\tilde{\phi}$\, and \,$\phi$\, we define the following sets\,:
\begin{align}
G^{\mu}_{\mathcal{I}}:= & \big\{ \phi \in G_{0} \ \ ; \ \exists \ \ \text{a lift} \
\ \tilde{\phi} \ \ \text{s.t.} \ \ \rho_{\mu}\big(\tilde{\phi}\big)=(0\,,0)  \big\}    \notag   \\
\tilde{G}^{\mu}_{\mathcal{I}}:=  &  \big\{ \tilde{\phi} \in \text{Homeo}_{0}(\mathbb{R}^{2}) \ \ ; \
\phi\in G_{0} \ \ \text{and} \ \  \rho_{\mu}\big(\tilde{\phi}\big)=(0\,,0) \big\}     \notag   \\
\tilde{G}_{\mathcal{I}}:=  &  \big\{ \tilde{\phi} \in \text{Homeo}_{0}(\mathbb{R}^{2}) \ \ ; \ \phi\in G_{0} \ \ \text{and}
\ \  \rho_{\nu}\big(\tilde{\phi}\big)=(0\,,0)   \notag   \ \
\text{for all} \ \  \nu \in {\mathcal P} (G) \big\}        \notag \\
= &   \bigcap_{\nu\in\mathcal{P}(G)} \!\! \tilde{G}^{\nu}_{\mathcal{I}} \,.
\notag
\end{align}

Since the map
\,$\rho_{\mu}:\text{Homeo}_{0,\mu}(\mathbb{R}^{2})\rightarrow\mathbb{R}^{2}$\,
is a morphism, it follows
that \,$\tilde{G}^{\mu}_{\mathcal{I}} \ , \ \tilde{G}_{\mathcal{I}}$\,
(resp. \,$G^{\mu}_{\mathcal{I}}$) are subgroups of \,$\text{Homeo}_{0}(\mathbb{R}^{2})$\,
(resp. \,$\text{Homeo}_{0}(\mathbb{T}^{2})$). Moreover, we have that
\,$\tilde{G}_{\mathcal{I}}$\,   and    \,$\tilde{G}^{\mu}_{\mathcal{I}}$\,
are lifts of
\,${G}_{\mathcal I}$\, and
\,${G}_{\mathcal I}^{\mu}$\, respectively, since the natural projections
\,$\tilde{\phi}\in\tilde{G}_{\mathcal I}^{\mu} \xrightarrow{\kappa} \phi \in {G}_{\mathcal I}^{\mu} $\, and
\,$\tilde{\phi}\in \tilde{G}_{\mathcal I} \xrightarrow{\kappa} \phi \in {G}_{\mathcal I}$\,  are isomorphisms.

Clearly, we also have \,$\mathcal{I}_{G} \subset G_{\mathcal{I}} \subset
G_{\mathcal{I}}^{\mu}$\, for any \,$\mu\in\mathcal{P}(G)$\,.
Notice that ${\mathcal P}(G)$ is non-empty
if $G$ is an amenable group.

In \cite{rib02}
we relate
the rotational properties of a nilpotent group with
the existence of lifts and we prove the following  proposition that will be used in the next section.
The condition  \,$1\notin\mathrm{spec}[\psi]$\, in next proposition is equivalent to
the non-vanishing of the Lefschetz number $L(\psi)$
(cf. \cite[Lemma 1.2]{rib02} \cite{Brooks-Brown-Pak-Taylor}).

\begin{pro}
\label{pro:bffi}
Let \,$G$\, be a nilpotent subgroup of \,$\mathrm{Homeo}({\mathbb T}^{2})$.
Suppose there exists \,$\psi \in G$\,  with \,$1\notin\mathrm{spec}[\psi]$\,.
Then we obtain
\,$\tilde{G}_{\mathcal I}=\tilde{G}_{\mathcal I}^{\mu}$\,
and \,$G_{\mathcal{I}} =G_{\mathcal I}^{\mu}$\, for any
\,$\mu \in {\mathcal P}(G)$.
Moreover \,$G_{\mathcal I}$\,
is  a finite index normal  subgroup of \,$G_{0}$\,.
\end{pro}

Nilpotent subgroups of \,$\mathrm{Homeo}({\mathbb T}^{2})$\, induce nilpotent
subgroups of the mapping class group of \,${\mathbb T}^{2}$, i.e.
nilpotent subgroups of \,$\mathrm{GL}(2\,,{\mathbb Z})$. We will need a classification
of such groups in order to study rotational properties.
They are virtually cyclic and metabelian.
Moreover, there exists a unique example of non-abelian group, up to conjugacy.

\vskip10pt
\begin{lem}
\label{lem:nmcg}
Let \,${\mathcal G}$\, be a nilpotent subgroup of
\,$\mathrm{MCG}({\mathbb T}^{2})$\,.
Then \,${\mathcal G}$\, is either of the form
\,$\langle N \rangle$\, or
\,$\langle N \,, -N \rangle$\, for some \,$N\in \mathcal{G}$\,,
or it is conjugated by a matrix in
\,$\mathrm{GL}(2\,,{\mathbb Q})$\, to the group
\[    {\mathcal H}:= \left\{
\left(
\begin{array}{rr}
1 & 0 \\
0 & 1 \\
\end{array}
\right), \
\left(
\begin{array}{rr}
-1 & 0 \\
0 & -1 \\
\end{array}
\right), \
\left(
\begin{array}{rr}
1 & 0 \\
0 & -1 \\
\end{array}
\right), \
\left(
\begin{array}{rr}
-1 & 0 \\
0 & 1 \\
\end{array}
\right), \right. \]
\[ \left.
\left(
\begin{array}{rr}
0 & -1 \\
1 & 0 \\
\end{array}
\right), \
\left(
\begin{array}{rr}
0 & 1 \\
-1 & 0 \\
\end{array}
\right), \
\left(
\begin{array}{rr}
0 & 1 \\
1 & 0 \\
\end{array}
\right), \
\left(
\begin{array}{rr}
0 & -1 \\
-1 & 0 \\
\end{array}
\right)
\right\}. \]
\end{lem}

The group ${\mathcal H}$ is isomorphic to the dihedral group $D_4$.
We admit subgroups of \,$\mathrm{MCG}({\mathbb T}^{2})$\, that contain orientation-reversing classes.
Indeed  ${\mathcal H}$ contains orientation-reversing classes. Hence a subgroup of
\,$\mathrm{MCG}({\mathbb T}^{2})$\, consisting of orientation-preserving classes is always abelian.
We  present a proof of the above lemma in the Appendix of \cite{rib02}.
%
%
%


\vskip30pt
\section{Construction of a $G$-invariant measure}
\label{subsec:consinv}

Let \,$G \subset \mathrm{Homeo}({\mathbb T}^{2})$\, be a nilpotent subgroup.
In this section we want to present a
description for the set
\,${\mathcal I}_{G}$\, of irrotational elements of \,$G$. In
particular we will show that \,${\mathcal I}_{G}$\, is a subgroup of \,$G$.
The main difficulty to prove this is that the sets \,${\mathcal P}(f)$\,
and \,${\mathcal P}(g)$\, can be very different for
\,$f,g \in G$. In order to overcome this problem, we do the following\,: given
\,$\phi \in G_{0}$\, and \,$\mu \in {\mathcal P}(\phi)$\,
we construct, under certain natural conditions,
a  measure \,$\mu' \in \mathcal{P}(G)$\,
such that \,${\rho}_{\mu}\big(\tilde{\phi}\big)={\rho}_{\mu'}\big(\tilde{\phi}\big)$\,
for any lift \,$\tilde{\phi}$\, of \,$\phi$.
This property can be interpreted as some rigidity property of  nilpotent groups
since a property that is satisfied for one element induces an analogous
global property.

Let us begin this section with a simple case, corresponding to groups of
isotopic to the identity homeomorphisms of the $2$-torus.
%
First we prove Theorem \ref{teo:main3}.

\vskip10pt
\begin{teo13}
\label{teo:rot0}

 If \,$G$\, is a nilpotent subgroup of \,$\mathrm{Homeo}({\mathbb T}^{2})$ then$\,:$
\begin{itemize}
\item[$(i)$]
 $G':=[\,G,G\,]$\, is irrotational when \,$G \subset \mathrm{Homeo}_{0}({\mathbb T}^{2})\,;$

\item[$(ii)$]
$G'':=[\,G',G'\,]$\,  is irrotational when  \,$G \subset \mathrm{Homeo}_{+}({\mathbb T}^{2})\,;$
\item[$(iii)$]
$G''':=[\,G'',G''\,]$\, is always irrotational.

\end{itemize}

\end{teo13}


In particular there are non-trivial irrotational elements in the nilpotent subgroups
of \,$\mathrm{Homeo}({\mathbb T}^{2})$\, except in simple cases.
For instance if \,$G \subset \mathrm{Homeo}_{0}({\mathbb T}^{2})$\, is nilpotent
then either \,${\mathcal I}_{G}$\, is non-trivial or \,$G$\, is abelian.

\begin{proof}[Proof of Theorem \ref{teo:main3}]
Suppose that \,$G$\, is a subgroup of \,$\mathrm{Homeo}_{0}({\mathbb T}^{2})$.
Furthermore suppose  \,$G$\,  finitely generated.
Let us remark from \cite{Yu,Austin}   that the
elements \,$\phi\in G'$\, grow sublinearly with respect to the word
metric in the following sense: \,$|\phi^{n}|=O(n^{1/2})$\, where
\,$|\phi^{n}|$\, denotes the distance from \,$\phi^{n}$\, to the
identity with respect to the word metric for a given finite
generating set of \,$G$.
In particular,  \,$\phi$\, is a distortion element i.e. the sequence
\,$\big(|\phi^{n}|/n \big)_{n\geq1}$\, goes to zero when \,$n$\,
tends to infinity. Moreover if \,$\phi^{n} =h_{i_{1}}
\circ h_{i_{2}} \circ \cdots \circ h_{i_{m}}$\, where \,$h_{1}, \ldots,
h_{s}$\, are generators of \,$G$, we choose lifts \,$\tilde{h}_{1},
\ldots, \tilde{h}_{s}, \tilde{\phi}$\, and we define
\,$\tilde{\psi}=\tilde{h}_{ i_{1}} \circ \cdots \circ \tilde{h}_{i_{m}}$.
We obtain
\[  (\tilde{\phi}^{n}(\tilde{x})- \tilde{x}) -  (\tilde{\phi}^{n}(\tilde{y}) -\tilde{y})
= (\tilde{\psi}(\tilde{x}) -\tilde{x}) - (\tilde{\psi}(\tilde{y}) - \tilde{y}) \]
\[ = \sum_{j=1}^{m} [(\tilde{h}_{i_{j}} - Id) \circ
\tilde{h}_{i_{j+1}} \circ \cdots \circ \tilde{h}_{i_{n}}] (\tilde{x}) -
\sum_{j=1}^{m} [(\tilde{h}_{i_{j}} - Id) \circ \tilde{h}_{i_{j+1}}
\circ \cdots \circ \tilde{h}_{i_{n}}](\tilde{y}). \]
We deduce
\,$\|(\tilde{\phi}^{n}(\tilde{x})-\tilde{x}) - (\tilde{\phi}^{n}(\tilde{y})-\tilde{y}) \| \leq 2 M |\phi^{n}|$\,
for all \,$\tilde{x}, \tilde{y} \in \mathbb{R}^{2}$\, where
\,$M = \max_{1 \leq i \leq s, \ \tilde{z}  \in {\mathbb
R}^{2}} \|\tilde{h}_{i}(\tilde{z}) -  \tilde{z} \|$\,
and \,$\|\cdot\|$\, denotes the usual norm in \,$\mathbb{R}^{2}$. Then we obtain
\begin{equation}
\label{equ:uniform} \lim_{n \to \infty} \max_{\tilde{x}, \tilde{y}}
\frac{\| (\tilde{\phi}^{n}(\tilde{x})  - \tilde{x}) - (\tilde{\phi}^{n}(\tilde{y}) - \tilde{y})
\|}{n} =0 \, .
\end{equation}
Let $\nu$ be a $\phi$-invariant ergodic Borel probability measure.
Birkhoff's ergodic theorem implies
\,$\lim_{n \to \infty} \frac{\tilde{\phi}^{n}(\tilde{x}) - \tilde{x}}{n}= \rho_{\nu}(\tilde{\phi})$\,
$\nu$-a.e. Moreover \,$\frac{\tilde{\phi}^{n}(\tilde{x}) - \tilde{x}}{n}$\,
converges uniformly to \,$\rho_{\nu}(\tilde{\phi})$\, when \,$n \to \infty$ by Equation (\ref{equ:uniform}).
In particular we obtain $\rho (\tilde{\phi}) = \{ \rho_{\nu}(\tilde{\phi}) \}$.
Analogously, these properties hold if \,$G$\, is any nilpotent subgroup
 of \,$\mathrm{Homeo}_{0}({\mathbb T}^{2})$\,
since any element of \,$G'$\, belongs to the derived group of some finitely
generated subgroup of \,$G$.

Since \,$G$\, is nilpotent there exists a probability measure
\,$\mu$\, that is invariant by all the elements of \,$G$.
Lemma \ref{lem:mor} implies that the map \,$\rho_{\mu}:
\mathrm{Homeo}_{0,\mu}({\mathbb R}^{2}) \to {\mathbb R}^{2}$\,  is a
morphism of groups.
The restriction of \,$\rho_{\mu}$\, to the derived group of
\,$\mathrm{Homeo}_{0,\mu}({\mathbb R}^{2})$\, is a vanishing map and hence
there exists a unique lift
\,$\tilde{\phi}$\, such that \,$\rho_{\mu}(\tilde{\phi})=(0,0)$\, for any
\,$\phi \in G'$. Since the rotation set $\rho (\tilde{\phi})$ contains exactly one point,
we deduce $\rho (\tilde{\phi}) = \{(0,0)\}$.
In particular \,$\phi$\, is irrotational and \,$\tilde{\phi}$\, is its
irrotational lift.

Suppose \,$G \subset \mathrm{Homeo}_{+}({\mathbb T}^{2})$.
In this case the group \,$[\,G\,]$\, is abelian by Lemma \ref{lem:nmcg}.
Hence \,$G'=[\,G,G\,]$\, is contained in \,$\mathrm{Homeo}_{0}({\mathbb T}^{2})$\,
and all the elements of \,$G''=[\,G',G'\,]$\, are irrotational.

If \,$G \subset \mathrm{Homeo}({\mathbb T}^{2})$\, then
 \,$G' \subset \mathrm{Homeo_{+}(\mathbb{T}^{2})}$\,
 and we have that \,$G''$\, is contained in  \,$\mathrm{Homeo}_{0}({\mathbb T}^{2})$. Consequently,  all the elements of \,$G'''=[\,G'',G''\,]$\, are irrotational.
%
\end{proof}

\vglue10pt

\begin{proof}[Proof of Theorem \ref{teo:main6} for \,$G \subset \mathrm{Homeo}_{0}({\mathbb T}^{2})$]
It suffices to show that given a lift \,$\tilde{\phi}$\, of an element \,$\phi$\, of \,$G$\, and
a \,$\phi$-invariant probability measure \,$\mu$\, then there exists a \,$G$-invariant probability
measure \,$\nu$\, such that \,$\rho_{\mu}(\tilde{\phi})=\rho_{\nu}(\tilde{\phi})$.

We will prove that there exists a \,$\langle Z^{(j)}(G), \phi \rangle$-invariant probability measure
\,$\nu_{j}$\, such that \,$\rho_{\mu}(\tilde{\phi})=\rho_{\nu_{j}}(\tilde{\phi})$\, by induction on \,$j$.
The result is obvious for \,$j=0$\, by defining \,$\nu_{0}=\mu$.
Suppose that the result holds true for \,$j$.
We have
\,$\rho_{\nu_{j}}(\tilde{\phi}) = \rho_{\psi^{*} \nu_{j}}\big(\tilde{\psi}^{-1} \circ \tilde{\phi} \circ \tilde{\psi}\big)$\,
for any lift \,$\tilde{\psi}$\, of any element \,$\psi  \in Z^{(j+1)}(G)$\, by  Lemma \ref{rem:pact}.
Notice that since \,$\psi$\, normalizes \,$\langle Z^{(j)}(G), \phi \rangle$\,, the measure
\,$\psi^{*} \nu_{j}$\, is \,$\langle Z^{(j)}(G), \phi \rangle$-invariant.
Consider a $G$-invariant Borel probability measure $\nu'$.
Since $\rho_{\nu'}: \mathrm{Homeo}_{0,\nu'}({\mathbb R}^{2}) \to {\mathbb R}^{2}$
is a morphism of groups by Lemma \ref{lem:mor}, we obtain
$\rho_{\nu'} [\tilde{\psi}^{-1},\tilde{\phi}] = (0,0)$.
Since  \,$G'$\, is irrotational by Theorem \ref{teo:main3} and
\,$[{\psi}^{-1},\phi] \in Z^{(j)}(G)$, we get
$\rho_{\psi^{*} \nu_{j}} [\tilde{\psi}^{-1},\tilde{\phi}]   = \rho_{\nu'} [\tilde{\psi}^{-1},\tilde{\phi}] = (0,0)$.
We deduce
\[ \rho_{\nu_{j}}(\tilde{\phi}) =
\rho_{\psi^{*} \nu_{j}}\big(\tilde{\psi}^{-1} \circ \tilde{\phi} \circ \tilde{\psi}\big) =
\rho_{\psi^{*} \nu_{j}} [\tilde{\psi}^{-1},\tilde{\phi}] + \rho_{\psi^{*} \nu_{j}} (\tilde{\phi}) =
\rho_{\psi^{*} \nu_{j}} (\tilde{\phi}) . \]
Hence \,$Z^{(j+1)}(G)$\, acts affinely on the compact convex set
\[
C_{j}:=\{ \nu \in {\mathcal P} \langle Z^{(j)}(G), \phi \rangle : \rho_{\nu}(\tilde{\phi})=\rho_{\mu}(\tilde{\phi})\}. \]
The fixed-point property of amenable groups implies that there exists a fixed point \,$\nu_{j+1}$\,
of the action of \,$Z^{(j+1)}(G)$\, on \,$C_{j}$\,, i.e. a \,$\langle Z^{(j+1)}(G), \phi \rangle$-invariant probability
measure \,$\nu_{j+1}$\,
such that \,$\rho_{\nu_{j+1}}(\tilde{\phi})=\rho_{\mu}(\tilde{\phi})$.
\end{proof}

\subsection{The construction method}
\label{subsec:cons}
We want to extend the result in Theorem \ref{teo:main6} to more general subgroups
of \,$\mathrm{Homeo}({\mathbb T}^{2})$. In order to achieve this goal, we construct
a \,$G$-invariant measure \,$\mu'$, following the classical
construction of an invariant measure for a solvable group of
homeomorphisms on a compact metrizable space.

Let \,$\phi \,,\mu$\, be as above. We denote \,$\rho_{0}= \rho_{\mu}(\tilde{\phi})$.
 Up to replace \,$\mu$\, with another
measure in \,${\mathcal P}(\phi)$\, we can suppose
\,$\mu \in {\mathcal P}(G_{0})$\, and \,$\rho_{\mu}(\tilde{\phi})=\rho_{0}$\, by Theorem \ref{teo:main6}
for $G_0 \subset  \mathrm{Homeo}_{0}({\mathbb T}^{2})$.
We construct new measures by using an inductive method.
We choose a finite set \,${\mathcal S}:= \{g_{1},\ldots,g_{n}\}$\, such that
\,$\langle G_{0}, g_{1}, \ldots, g_{n} \rangle = G$.
If we do not require any other property to the elements of \,${\mathcal S}$\, then
we can choose
\,$n \leq 2$\, by Lemma \ref{lem:nmcg}.
Let us define
\,$\mu_{0} = \mu$. Our goal is constructing a Borel probability measure
\,$\mu_{j}$\, that is $\psi$-invariant for any
\,$\psi \in \langle G^{j}, \phi \rangle$\, where
\,$G^{j}=\langle G_{0}, g_{0}\,,g_{1},\ldots,g_{j}\rangle$\, and
\,$g_{0}$\, denotes the identity map of \,$\mathbb{T}^{2}$.
In this context the following condition will play a fundamental role in the construction\,:
\[ (*): \ \mathrm{spec}[g_{j}] \cap (e^{2 \pi i {\mathbb Q}} \setminus \{1\}) = \emptyset \quad
\text{for all} \quad 0 \leq j \leq n \]
where \,$e^{2 \pi i {\mathbb Q}}$\, is the set of roots of the unity.
We will see that Condition $(*)$ and the construction process implies
\,$\rho_{\mu_{j}}\big(\tilde{\phi}\big) = \rho_{\mu}\big(\tilde{\phi}\big)$\,
for all lifts \,$\tilde{\phi}$\, of \,$\phi$\, and  \,$0 \leq j \leq n$.
We say that \,$G$\, satisfies the Condition $(*)$ if there exists a set \,${\mathcal S}$\,
as described above. Let us suppose that  \,$G/G_{0}$\,  is abelian.
Otherwise \,$G/G_{0}$\, is conjugated to \,${\mathcal H}$ by Lemma \ref{lem:nmcg}
and it is easy to see that \,$G$\, does not satisfy the Condition $(*)$.


\vglue10pt

Let us suppose  that the \,$\langle G^{j}, \phi \rangle$-invariant measure \,$\mu_{j}$\,
has been already constructed for \,$0 \leq j < n$.
Then we define the measures\,:
\[ \mu_{j,\ell} = (g_{j+1}^{\ell})_{*} (\mu_{j}) \quad \mathrm{for} \ \,\ell \in {\mathbb Z}
\quad  \mathrm{and} \quad \nu_{j,\ell} =
\frac{1}{\ell} \, {\sum_{p=0}^{\ell-1} \mu_{j,p}}
\quad \mathrm{for} \ \,\ell \in {\mathbb Z}^{+} . \]
From the compactness of the space of probability measures on \,${\mathbb T}^{2}$\,
in the weak$^{*}$ topology we know that
the sequence \,$\big\{\nu_{j,\ell}\big\}_{\ell\geq1}$\, has a subsequence that converges to a probability
measure \,$\mu_{j+1}$\, in this topology. Of course
the choice of \,$\mu_{j+1}$\, is not necessarily
unique. The measure \,$\mu_{j+1}$\, is \,$g_{j+1}$-invariant since we have
the relation\,:
$$(g_{j+1})_{*}(\nu_{j,\ell}) - \nu_{j,\ell}=\frac{1}{\ell}
\Big(\mu_{j,\ell}- \mu_{j,0} \Big) \quad
\text{for all} \quad \ell\in\mathbb{Z}^{+}.$$

Let us check out that \,$\mu_{j,\ell}$\, is \,$\langle G^{j}, \phi \rangle$-invariant for any \,$\ell \in {\mathbb Z}$.
For this let us fix \,$\psi \in \langle G^{j}, \phi \rangle$.
Since   $G/G_{0}$ is abelian, we have
\,$\psi^{-1} \circ g_{j+1}^{-\ell} \circ \psi \circ g_{j+1}^{\ell}=h_{j}$\, for some
\,$h_{j} \in G_{0} \cap G'$\,
and we conclude that
\[ \psi_{*} (\mu_{j,\ell}) =
\psi_{*} \big((g_{j+1}^{\ell})_{*} (\mu_{j})\big) = (\psi \circ g_{j+1}^{\ell})_{*} (\mu_{j})  =
(g_{j+1}^{\ell} \circ \psi  \circ h_{j})_{*} (\mu_{j})  \]
\[ = \big((g_{j+1}^{\ell})_{*} \circ  \psi_{*} \circ (h_{j})_{*}\big)(\mu_{j}) = (g_{j+1}^{\ell})_{*}( \psi_{*}(\mu_{j})) =
(g_{j+1}^{\ell})_{*} (\mu_{j}) = \mu_{j,\ell} \]
for any \,$\ell \in {\mathbb Z}$. We obtain that \,$\nu_{j,\ell}$\,
is \,$\langle G^{j}, \phi \rangle$-invariant for any \,$\ell \in {\mathbb Z}^{+}$.
Thus, the measure \,$\mu_{j+1}$\, is \,$g_{j+1}$-invariant and
\,$\langle G^{j}, \phi \rangle$-invariant.
Consequently,  it is
\,$\langle G^{j+1}, \phi \rangle$-invariant.

\vglue10pt

 Now, let us remember some elementary properties about the eigenvalues of the elements
 of \,$\mathrm{GL}(2\,,\mathbb{Z})$\,.

\vskip10pt
\begin{rem}\label{rem:spec:01}
 For a given \,$A \in \mathrm{GL}(2\,,\mathbb{Z})$\, we have\,:
\begin{itemize}
\item
 The complex (non real) eigenvalues occur when \,$\mathrm{tr}(A)=0\,,\pm1$\,
and \,$\det(A)=1$. They are\,: \,$\pm i$\, (if \,$\mathrm{tr}(A)=0$\,) ,
\,$\frac{1}{2}\pm \frac{\sqrt{3\,}}{2} i$\, (if \,$\mathrm{tr}(A)=1$\,)
or \,$-\frac{1}{2}\pm \frac{\sqrt{3\,}}{2} i$\, (if \,$\mathrm{tr}(A)=-1$\,)\,;

\item
 The eigenvalue \,$-1$\, occurs when \,$\mathrm{tr}(A)=-2$\, and
\,$\det(A)=1$\, or when \,$\mathrm{tr}(A)=0$\, and
\,$\det(A)=-1$\,;

\item
 The eigenvalue \,$1$\, occurs when \,$\mathrm{tr}(A)=2$\, and
\,$\det(A)=1$\, (i.e. when \,$A$\, is either the identity map or a  Dehn twist) or when
\,$\mathrm{tr}(A)=0$\, and
\,$\det(A)=-1$\,;

\item
 In all the other cases \,$A$\, has an eigenvalue of modulus greater then
\,$1$\, and another of modulus less then \,$1$.

\end{itemize}

\end{rem}
\vskip5pt

\vskip20pt
\subsection{Properties of the measure}


 The idea of the construction
is trying to prove that, defined the measures \,$\mu_{0},\mu_{1},\ldots,\mu_{j}$\, as above,  we define
 the subsequent measures in the process preserving rotation numbers
of the elements of \,$G^{j}$. For this, first we need to obtain a formula for rotation numbers
with respect to the new measures as a function of the old ones.
Following our convention \,$\tilde{\psi}$\, denotes a lift of \,$\psi \in \textrm{Homeo}(\mathbb{T}^{2})$\,
to the universal covering of \,$\mathbb{T}^{2}$.

\vskip10pt
\begin{lem}
\label{lem:rotev}
Suppose that
\begin{equation}
\label{equ:rotd}
 \rho_{\mu_{j,p}}\big(\tilde{h}\big) = \rho_{\mu_{j}}\big(\tilde{h}\big) = \rho_{\mu_{n}}\big(\tilde{h}\big)
\end{equation}
for all $0 \leq j < n$, $p \in {\mathbb Z}$ and $h \in Z^{(k)}(G) \cap  G_{0}$.
Then we have
\begin{equation}
\label{equ:pre}
\rho_{\mu_{j,p}}\big(\tilde{h}\big) =
[g_{j+1}]^{p}\big(\rho_{\mu_{j}}\big(\tilde{h}\big)\big) + \sum_{k=1}^{p} \, [g_{j+1}]^{k}
\big(\rho_{\mu_{j}}\big(\big[\,\tilde{g}_{j+1}^{-1} \,, \tilde{h}\hskip1pt\big]\big)\big)
\end{equation}
for all \,$0 \leq j <n$ , $p \geq 1$ , $h \in Z^{(k+1)}(G) \cap G_{0}$\, and
\begin{equation}
\label{equ:pre2}
\rho_{\mu_{j,p}}\big(\tilde{h}\big) =
[g_{j+1}]^{p}\big(\rho_{\mu_{j}}\big(\tilde{h}\big)\big) - \sum_{k=1}^{-p} \, [g_{j+1}]^{-(k-1)}
\big(\rho_{\mu_{j}}\big(\big[\,\tilde{g}_{j+1}^{-1} \,, \tilde{h}\hskip1pt\big]\big)\big)
\end{equation}
for all \,$0 \leq j <n$ , $p \leq -1$ , $h \in Z^{(k+1)}(G) \cap G_{0}$.
\end{lem}
\begin{proof}
Fix \,$0 \leq j < n$. First we will
show Equation \eqref{equ:pre}   by induction on \,$p$.
The case \,$p=1$\, is a consequence of  Lemmas \ref{rem:pact} and \ref{lem:mor};
indeed we obtain
\begin{align}
\rho_{\mu_{j,1}}\big(\tilde{h}\big)  &   =
\rho_{(g_{j+1})_{*}(\mu_{j})}\big(\tilde{h}\big)=[g_{j+1}]\big(\rho_{\mu_{j}}\big(
\tilde{g}_{j+1}^{-1}\circ \tilde{h} \circ \tilde{g}_{j+1}\big)\big) \notag \\
 &  =  [g_{j+1}]\big(\rho_{\mu_{j}}\big(\tilde{h}\big)\big) +
 [g_{j+1}]\big(\rho_{\mu_{j}}\big(
[\,\tilde{g}_{j+1}^{-1} \,, \tilde{h}\,]\big)\big) \notag
\end{align}
for any \,$h \in Z^{(k+1)}(G) \cap G_{0}$\,.
Let us remind the reader
that \,$(g^{m}_{j+1})_{*}(\mu_{j})$\, is \,$\langle G_{0}, \phi\rangle$-invariant
for all \,$m\in\mathbb{Z}$\,. Moreover, if
\,$h \in Z^{(k+1)}(G) \cap G_{0}$\, then
\,$[\,g^{-1}_{j+1}\,,h\,] \in G_{0}$\, since
$G/G_{0}$ is abelian.
Suppose the result holds for some integer \,$p\geq 1$\,. Given
\,$h \in Z^{(k+1)}(G) \cap G_{0}$\, we have
\begin{align}
 \rho_{\mu_{j,p+1}}\big(\tilde{h}\big)  &   =
\rho_{(g_{j+1}^{p+1})_{*}(\mu_{j})}\big(\tilde{h}\big)=[g_{j+1}]\big(\rho_{(g_{j+1}^{p})_{*}(\mu_{j})}\big(
\tilde{g}_{j+1}^{-1}\circ \tilde{h} \circ \tilde{g}_{j+1}\big)\big) \notag \\
 &  =  [g_{j+1}]\big(\rho_{(g_{j+1}^{p})_{*}(\mu_{j})}\big([\,\tilde{g}_{j+1}^{-1} \,, \tilde{h}\,]\big)\big) +
 [g_{j+1}]\big(\rho_{(g_{j+1}^{p})_{*}(\mu_{j})}\big(\tilde{h}\big)\big) \notag   \\
    &  =  [g_{j+1}]\big(\rho_{\mu_{j,p}}\big([\,\tilde{g}_{j+1}^{-1} \,, \tilde{h}\,]\big)\big) +
 [g_{j+1}]\big(\rho_{\mu_{j,p}}\big(\tilde{h}\big)\big)  \notag
\end{align}
by Lemmas \ref{rem:pact} and \ref{lem:mor}. The induction hypothesis gives the
equality
\begin{align}
\rho_{\mu_{j,p+1}}\big(\tilde{h}\big) & =
[g_{j+1}]\Big([g_{j+1}^{p}]\big(\rho_{\mu_{j}}\big(\tilde{h}\big)\big) + \sum_{k=1}^{p}
\, [g_{j+1}^{k}]\big(\rho_{\mu_{j}}\big([\,\tilde{g}_{j+1}^{-1} \,, \tilde{h}\,]\big)\big)\Big) \notag  \\
& \hskip10pt + [g_{j+1}]\big(\rho_{\mu_{j,p}}\big([\,\tilde{g}_{j+1}^{-1} \,, \tilde{h}\,]\big)\big) \notag  \\
&  =  [g_{j+1}^{p+1}]\big(\rho_{\mu_{j}}\big(\tilde{h}\big)\big) +
[g_{j+1}]\big(\rho_{\mu_{j,p}}\big([\,\tilde{g}_{j+1}^{-1} \,, \tilde{h}\,]\big)\big) +
\sum_{k=2}^{p+1}
\, [g_{j+1}^{k}]\big(\rho_{\mu_{j}}\big([\,\tilde{g}_{j+1}^{-1} \,, \tilde{h}\,]\big)\big). \notag \end{align}
Since \,$[\,g_{j+1}^{-1}\,,h\,] \in Z^{(k)}(G) \cap G_{0}$\, it follows from hypothesis (\ref{equ:rotd}) that
$$\rho_{\mu_{j,p}}\big([\,\tilde{g}_{j+1}^{-1} \,, \tilde{h}\,]\big)=
\rho_{\mu_{j}}\big([\,\tilde{g}_{j+1}^{-1} \,, \tilde{h}\,]\big).$$
Thus, returning to the last expression of
\,$\rho_{\mu_{j,p+1}}\big(\tilde{h}\big)$\, we obtain Equation \eqref{equ:pre} for the case
\,$p+1$\, and the proof of \eqref{equ:pre} is finished.

Now, we apply Equation (\ref{equ:pre}) to \,$g_{j+1}^{-1}$\, and we obtain\,:
\begin{align}\label{apli:cas:ant}
\rho_{\mu_{j,p}}\big(\tilde{h}\big) =
[g_{j+1}]^{p}\big(\rho_{\mu_{j}}\big(\tilde{h}\big)\big) + \sum_{k=1}^{-p} \, [g_{j+1}]^{-k}
\big(\rho_{\mu_{j}}\big(\big[\,\tilde{g}_{j+1} \,,
\tilde{h}\hskip1pt\big]\big)\big)
\end{align}
for all \,$p\leq -1$\, and
\,$h\in  Z^{(k+1)}(G) \cap G_{0}$\,. Returning to the hypothesis
\eqref{equ:rotd} and remembering that \,$[\,{g}_{j+1}\,,{h}\,]\in Z^{(k)}(G) \cap G_{0}$\, we
conclude by using  Lemmas \ref{rem:pact} and \ref{lem:mor} that\,:

\begin{align}
 \rho_{\mu_{j}}\big([\,\tilde{g}_{j+1} \,, \tilde{h}\,]\big)
 & =  \rho_{\mu_{n}}\big(\tilde{g}_{j+1} \circ \tilde{h} \circ \tilde{g}_{j+1}^{-1}
 \circ \tilde{h}^{-1} \big)=
 \rho_{\mu_{n}}\big(\tilde{g}_{j+1} \circ \tilde{h} \circ \tilde{g}_{j+1}^{-1}
 \big) -  \rho_{\mu_{n}}\big(\tilde{h} \big) \notag  \\
 & = -[g_{j+1}]\Big( [g_{j+1}^{-1}]\big( \rho_{\mu_{n}}\big(\tilde{h} \big)\big) -
[g_{j+1}^{-1}] \big(\rho_{\mu_{n}}\big(\tilde{g}_{j+1} \circ \tilde{h} \circ \tilde{g}_{j+1}^{-1}
 \big)\big)\Big)     \notag  \\
 & = - [g_{j+1}]\Big( \rho_{\mu_{n}}\big(\tilde{g}_{j+1}^{-1} \circ \tilde{h} \circ \tilde{g}_{j+1}
 \big)  -   \rho_{\mu_{n}}\big(\tilde{h} \big)\Big)   \notag  \\
 & =   - [g_{j+1}]\big( \rho_{\mu_{n}}\big([\,\tilde{g}_{j+1}^{-1} \,, \tilde{h}\,] \big) \big)=
  - [g_{j+1}]\big( \rho_{\mu_{j}}\big([\,\tilde{g}_{j+1}^{-1} \,, \tilde{h}\,] \big) \big) \,.      \notag
\end{align}
From the above equality and Equation \eqref{apli:cas:ant} we conclude Equation \eqref{equ:pre2}.
\end{proof}

Rotation numbers can vary along the process if we do not impose any condition.
Roughly speaking Condition $(*)$ implies
that the sequence of rotation numbers is either constant or unbounded.
We obtain invariance of rotation numbers since clearly the latter situation is impossible.

\vskip10pt

\begin{lem}
\label{lem:rinv}
Consider
\,$\{g_{1},\ldots,g_{n}\}$\,  satisfying Condition $(*)$\,.
Fix \,$k \geq 0$\, and  suppose
\begin{equation}
\label{equ:que1}
\rho_{\mu_{j}}({\tilde{g}}) =
\rho_{\mu_{j \!,q}}({\tilde{g}}) =
\rho_{\nu_{j \!,p}}({\tilde{g}})
= \rho_{\mu_{n}}({\tilde{g}})
\end{equation}
for all \,$0 \leq j <n$ \ , \ $q \in {\mathbb Z}$ \ , \
$p \in {\mathbb Z}^{+}$\,
and \,$g \in Z^{(k)}(G) \cap G_{0}$\,.
Then we obtain
\begin{equation}
\label{equ:que2}
\rho_{\mu_{j}}({\tilde{h}}) =
\rho_{\mu_{j \!,q}}({\tilde{h}}) =
\rho_{\nu_{j \!,p}}({\tilde{h}})
= \rho_{\mu_{n}}({\tilde{h}})
\end{equation}
for all \,$0 \leq j <n$ \ , \ $q \in {\mathbb Z}$ \ , \
$p \in {\mathbb Z}^{+}$\,
and \,$h \in Z^{(k+1)}(G) \cap G_{0}$\,.
\end{lem}
\vskip5pt

\begin{proof}
Fix \,$0 \leq j < n$\, and let \,$h \in Z^{(k+1)}(G) \cap G_{0}$\,.
We will prove that
\,$ \rho_{\mu_{j,q}}(\tilde{h}) = \rho_{\mu_{j}}(\tilde{h})$\,
for any \,$q \in {\mathbb Z}$. For this, we recall
that the hypothesis on \,$[g_{j+1}]$\, implies that the
mapping class \,$[g_{j+1}]$\, is either trivial, equal to a
Dehn twist or \,$1 \notin \text{spec}[g_{j+1}]$\, by Remark \ref{rem:spec:01}.

Suppose that \,$[g_{j+1}]=Id$. Notice that
\[\rho_{\mu_{j,q}}\big(\tilde{h}\big) =
\rho_{\mu_{j}}\big(\tilde{h}\big) + q
\rho_{\mu_{j}}\big(\big[\,\tilde{g}_{j+1}^{-1} \,, \tilde{h}\hskip1pt\big]\big) \]
for all $h \in Z^{(k+1)}(G) \cap G_{0}$ and $q \in {\mathbb Z}$ by Equations
(\ref{equ:pre}) and (\ref{equ:pre2}). Since the set of rotation vectors of $\tilde{h}$ is bounded,
it follows $\rho_{\mu_{j}}\big([\,\tilde{g}_{j+1}^{-1}\,,\tilde{h}\,]\big)=(0,0)$.
Thus $\rho_{\mu_{j,q}}\big(\tilde{h}\big) =\rho_{\mu_{j}}\big(\tilde{h}\big)$ for any $q \in {\mathbb Z}$.

Suppose that \,$[g_{j+1}]$\, is a (non-trivial) Dehn twist.
Up to a change of coordinates in
\,$\mathrm{GL}(2\,,{\mathbb Z})$\, we can suppose
\[ [g_{j+1}] = \left(
\begin{array}{cc}
1& 0 \\
m & 1 \\
\end{array}
\right)   \quad \text{for some} \quad 0 \neq m \in {\mathbb Z}. \]
Thus, in these new coordinates, \,$\rho_{\mu_{j,q}}\big(\tilde{h}\big)$\, takes the form\,:
\[ \rho_{\mu_{j,q}}\big(\tilde{h}\big) = \big(v_{1}, m qv_{1} + v_{2}\big) +
\big(q w_{1}, (1 + \cdots + q) m w_{1} +q w_{2}\big) \quad \text{for all}
\quad q \in {\mathbb Z}^{+} \]
where \,$\rho_{\mu_{j}}\big(\tilde{h}\big)=(v_{1}\,,v_{2})$\, and
\,$\rho_{\mu_{j}}\big([\,\tilde{g}_{j+1}^{-1}\,,\tilde{h}\,]\big)=(w_{1}\,,w_{2})$\,
by Lemma \ref{lem:rotev}.
Since the set of rotation vectors of \,$\tilde{h}$\, is bounded,
we deduce \,$w_{1}=0$\, and
\[ \rho_{\mu_{j,q}}(\tilde{h}) = (v_{1}, q(m v_{1} + w_{2}) + v_{2}) \quad \text{for all}
\quad q \in {\mathbb Z}^{+}. \]
Again, the boundness of the set of rotation vectors of \,$\tilde{h}$\, guarantee that  \,$m v_{1} + w_{2}=0$\,  and then  we have
\,$\rho_{\mu_{j,q}}(\tilde{h}) = \rho_{\mu_{j}}(\tilde{h})$\,
for any \,$q \in {\mathbb Z}^{+}$. Similar arguments with Equation (\ref{equ:pre2})
give us \,$\rho_{\mu_{j,q}}(\tilde{h}) = \rho_{\mu_{j}}(\tilde{h})$\, for any integer \,$q<0$\, and then we obtain \,$\rho_{\mu_{j,q}}(\tilde{h}) = \rho_{\mu_{j}}(\tilde{h})$\, for all \,$q\in \mathbb{Z}$.

Suppose finally that \,$1 \not \in \mathrm{spec}[g_{j+1}]$\,.
From Proposition \ref{pro:bffi} there exists \,$\tau \in \mathbb{Z}^{+}$\, such that
\,$h^{\tau} \in G_{\mathcal{I}}=G^{\mu_{n}}_{\mathcal{I}}$\, and hence
\,$\rho_{\mu_{n}}\big(\tilde{h}^{\tau}\big) \in \mathbb{Z}^{2}$. On the other hand
\,$[\,{g}_{j+1}^{-1}\,,{h}\,] \in Z^{(k)}(G) \cap G_{0}$\, and we conclude
\begin{align}\label{equ:tri:01}
\rho_{\mu_{j}}\big([\,\tilde{g}_{j+1}^{-1}\,,\tilde{h}\,]\big)
& = 
\rho_{\mu_{n}}\big(\tilde{g}_{j+1}^{-1} \circ \tilde{h} \circ \tilde{g}_{j+1}
\circ \tilde{h}^{-1} \big) \notag \\
& = [g_{j+1}]^{-1}\big( \rho_{\mu_{n}}\big(\tilde{h}\big)\big) -
\rho_{\mu_{n}}\big(\tilde{h}\big)  
\end{align}
where the first equality comes from \eqref{equ:que1} whereas the second is consequence of
Lemmas \ref{rem:pact} and \ref{lem:mor}.
Now, let \,$u=\rho_{\mu_{n}}\big(\tilde{h}^{\tau}\big)$. Replacing \,$\tilde{h}$\, with
\,$T_{-u}\circ\tilde{h}^{\tau}$\, in \eqref{equ:tri:01} we obtain
\begin{align}
\rho_{\mu_{j}}\big([\,\tilde{g}_{j+1}^{-1}\,,T_{-u}\circ\tilde{h}^{\tau}\,]\big)
& = [g_{j+1}]^{-1}\big( \rho_{\mu_{n}}\big(T_{-u}\circ\tilde{h}^{\tau}\big)\big) -
\rho_{\mu_{n}}\big(T_{-u}\circ\tilde{h}^{\tau}\big)  \notag \\
&  =  [g_{j+1}]^{-1}\big( -u+\rho_{\mu_{n}}\big(\tilde{h}^{\tau}\big)\big)
+  \big(u- \rho_{\mu_{n}}\big(\tilde{h}^{\tau}\big)\big)=(0\,,0). \notag
\end{align}
By plugging the last expression in Equations
\eqref{equ:pre} and \eqref{equ:pre2} we obtain
$$\rho_{\mu_{j,q}}\big(T_{-u}\circ\tilde{h}^{\tau}\big)=[g_{j+1}]^{q}
\big(\rho_{\mu_{j}}\big(T_{-u}\circ\tilde{h}^{\tau}\big)\big)
\quad \text{for all} \quad q\in \mathbb{Z}.$$
Again, the boundness of the set of rotation vectors \,$\rho(T_{-u}\circ \tilde{h}^{\tau})$\, allow us to deduce that the set
\,$\big\{ [g_{j+1}]^{q}
\big(\rho_{\mu_{j}}\big(T_{-u}\circ\tilde{h}^{\tau}\big)\big)\big\}_{q\in\mathbb{Z}}$\,
is bounded.

 Moreover, Remark \ref{rem:spec:01} and the  hypothesis on
\,$\mathrm{spec}[g_{j+1}]$\, imply that
\,$[g_{j+1}]$\, has
an eigenvalue of modulus greater than \,$1$\, and
an eigenvalue of modulus less than \,$1$\,.
These properties are compatible only if
\,$\rho_{\mu_{j}}\big(T_{-u}\circ\tilde{h}^{\tau}\big) =0$\,.
In such a case we have \,$\rho_{\mu_{j,q}}\big(T_{-u}\circ\tilde{h}^{\tau}\big)=
\rho_{\mu_{j}}\big(T_{-u}\circ\tilde{h}^{\tau}\big)$\, and then
$$\rho_{\mu_{j,q}}\big(\tilde{h}\big)=
\rho_{\mu_{j}}\big(\tilde{h}\big) \quad \text{for any} \quad q\in \mathbb{Z}\,.$$

As a consequence we obtain \,$\rho_{\nu_{j,p}}\big(\tilde{h}\big) =
\rho_{\mu_{j}}\big(\tilde{h}\big)$\,
for any \,$p \in \mathbb{Z}^{+}$\,
and then \,$\rho_{\mu_{j+1}}\big(\tilde{h}\big) =\rho_{\mu_{j}}\big(\tilde{h}\big)$.
We are done by making \,$j$\, vary in \,$\{0,\ldots,n-1\}$.
\end{proof}

\vskip10pt

Next we identify the nilpotent subgroups of \,$\mathrm{MCG}({\mathbb T}^{2})$\,
that are compatible with the existence of generator sets satisfying
Condition $(*)$.
First we consider the following Condition $(**)$ on
\,$G \subset \textrm{Homeo}(\mathbb{T}^{2})$\,:

\vskip5pt

$(**)$ \quad \parbox[]{.88\textwidth}{%
$[G]$\, has one of the  following forms\,:
\,$[G]=\{Id\}$ \ , \ $[G]=\langle D \rangle$\, where \,$D$\, is a (non-trivial) Dehn twist \ , \
\,$[G]=\langle A \rangle$ \ or \ $[G]= \langle A \,,-Id \, \rangle$ \ where
\ $\mathrm{spec}(A)$\, does not contain any eigenvalue of modulus \,$1$.
}

\vskip5pt

Now we state the following remark.

\vskip10pt

\begin{rem}\label{cen:ger:set:02}
Let \,$G \subset \textrm{Homeo}(\mathbb{T}^{2})$\, be a nilpotent group
and suppose that \,$G$\, satisfies Condition $(**)$.
We can construct  a set \,${\mathcal S}:=\{g_{1},\ldots,g_{n}\}$\,
such that \,$G = \langle G_{0}, g_{1}, \ldots, g_{n} \rangle$\, and \,${\mathcal S}$\, satisfies
Condition $(*)$.

The remark is trivial if \,$[G]=\{Id\}$\, or
\,$[G]=\langle D \rangle$\, where \,$D$\, is a (non-trivial) Dehn twist
since then \,$\textrm{spec}[g]=\{1\}$\, for all \,$g\in G$.

\vglue5pt
\noindent
Let us consider the case \,$[G]=\langle A \rangle$.
It suffices to define \,$n=1$\, and to take \,$g_{1} \in G$\, such that
\,$[g_{1}]=A$.

\vglue5pt
\noindent
Consider now the case \,$[G]=\langle A\,, -Id \rangle$.
We define \,$n=2$\, and elements \,$g_{1}, g_{2} \in G$\, such that \,$[g_{1}]=A$\, and
\,$[g_{2}]=-A$.
\end{rem}

\vskip5pt

In the next remark we describe the  nilpotent subgroups
of \,$\mathrm{Homeo}({\mathbb T}^{2})$\,
that do not satisfy Condition $(**)$.

\vskip5pt
\begin{rem}
\label{rem:big}
 Let \,$G$\, be a nilpotent subgroup
of \,$\mathrm{Homeo}({\mathbb T}^{2})$.
 Following Remark \ref{rem:spec:01} it is very simple to conclude that if
\,$[G]\neq \{Id\}$\, is finite then \,$G$\, does not satisfy neither Condition $(*)$ nor Condition $(**)$.
Suppose that \,$[G]$\,  is an infinite group that does not satisfy Condition $(**)$.
It is of the form \,$\langle N \rangle$\, or
\,$\langle N ,-Id \rangle$\ by Lemma \ref{lem:nmcg}, where \,$Id \neq N \in \mathrm{GL}(2\,,\mathbb{Z})$\,
has an eigenvalue of modulus equal to \,$1$.
Notice that the eigenvalues of \,$N$\, are real numbers, otherwise
\,$[G]$\, is finite by Remark \ref{rem:spec:01}. Analogously we deduce \,$\text{spec}(N) \neq  \{1\,,-1\}$.
We obtain
%
%
%
\begin{itemize}
\item
If \,$\text{spec}(N)=\{1\}$\,  then \,$N$\, is a Dehn twist\,;
\item
If \,$\text{spec}(N)=\{-1\}$\,  then   \,$N=-D$\, where \,$D$\, is a Dehn twist.
\end{itemize}

 Consequently, the nilpotent subgroups \,$G$\, of \,$\mathrm{Homeo}({\mathbb T}^{2})$\,
that do not satisfy Condition $(**)$
have one of the forms:
\,$[G]$\, is a non-trivial finite group, \,$[G]= \langle -D \rangle$\,
or \,$[G]=\langle D,-Id \rangle$\, for a (non-trivial) Dehn twist \,$D$.

In these cases the set \,$S$\, of elements of \,$[G]$\, that do not have
non-trivial roots of the unit as eigenvalues is a proper finite index normal
subgroup of \,$[G]$. Hence \,$G$\, does not satisfy
Condition $(*)$ and our construction can not be applied. The previous discussion, together
with Remark \ref{cen:ger:set:02}, imply that
Condition $(*)$ and Condition $(**)$ are equivalent.

Notice that \,$S=[G]$\, for all the groups that satisfy Condition $(**)$
except \,$[G]=\langle A \,,-Id \rangle$\, where \,$S=[G] \setminus \{-Id\}$\,
and \,$\langle S \rangle = [G]$.
\end{rem}
\vskip5pt

\begin{pro}
\label{pro:bigsma}
Let \,$G$\, be a
nilpotent subgroup of \,$\mathrm{Homeo}({\mathbb T}^{2})$\,
and suppose that \,$G$\, satisfies Condition $(**)$\,.
Given \,$\phi \in G_{0}$\, and  \,$\mu \in \mathcal{P}(\phi)$\,
there exists \,$\nu \in {\mathcal P}(G)$\, such that
\,$\rho_{\mu}(\tilde{\phi})=\rho_{\nu}(\tilde{\phi})$\, for any lift \,$\tilde{\phi}$\,
of \,$\phi$\,. Moreover if $\mu$ is $G_{0}$-invariant we have
$\rho_{\mu}(\tilde{\psi})= \rho_{\nu}(\tilde{\psi})$\, for any lift \,$\tilde{\psi}$\,
of every \,$\psi \in G_{0}$\,.
\end{pro}

\begin{proof}
We can suppose that \,$\mu$\, is \,$G_{0}$-invariant by Theorem \ref{teo:main6}
for subgroups of
$\mathrm{Homeo}_{0}({\mathbb T}^{2})$.

Since $G$ satisfies Condition $(**)$, it follows that $[G]$ is
infinite by Remark \ref{rem:big} and thus abelian by Lemma \ref{lem:nmcg}.
Then we can apply
the construction of a $G$-invariant measure in subsection \ref{subsec:cons}.
There exists a
subset ${\mathcal S}:= \{g_{1}, \hdots, g_{n}\}$ of $G$ such that
$G = \langle G_{0}, g_{1},\hdots,g_{n} \rangle$ and ${\mathcal S}$ satisfies
Condition $(*)$ by Remark \ref{cen:ger:set:02}.
Equation (\ref{equ:que1}) in Lemma \ref{lem:rinv} holds trivially for $k=0$.
Moreover if it is satisfied for $k$ then so it is for $k+1$ by Lemma \ref{lem:rinv}.
We deduce
\,$\rho_{\mu}(\tilde{\psi})= \rho_{\nu}(\tilde{\psi})$\, for any lift \,$\tilde{\psi}$\,
of any \,$\psi \in G_{0}$\, where \,$\nu:= \mu_{n}$\, is the \,$G$-invariant
probability measure provided by Lemma \ref{lem:rinv}.
\end{proof}


 Now, we compare our definitions of irrotationality
for nilpotent groups that satisfy Condition $(*)$.

\vskip10pt

\begin{pro}
\label{pro:bigsma2}
Let \,$G$\, be a nilpotent subgroup of \,$\mathrm{Homeo}({\mathbb T}^{2})$\,
and suppose that \,$G$\, satisfies Condition $(**)$\,.
%
%
%
Then we have \,${\mathcal I}_{G}=G_{\mathcal I}$.
\end{pro}

\begin{proof}
It is clear that \,${\mathcal I}_{G} \subset G_{\mathcal I}$.
Now, suppose there exists  \,$\phi \in G_{\mathcal{I}}$\, such that
\,$\phi \notin \mathcal{I}_{G}$.
Fix a lift \,$\tilde{\phi}$\, of \,$\phi$. Since
\,$\phi$\, is non-irrotational, the homeomorphism \,$\tilde{\phi}$\, is not
an irrotational lift of \,$\phi$. Thus there exists \,$\mu \in \mathcal{P}(\phi)$\,
such that \,$\rho_{\mu}(\tilde{\phi}) \neq (0,0)$.
There exists \,$\nu \in \mathcal{P}(G)$\, such that
\,$\rho_{\nu}\big(\tilde{\phi}\big)=\rho_{\mu}\big(\tilde{\phi}\big)\neq (0\,,0)$\,
by Proposition \ref{pro:bigsma}.

Thus we deduce that \,$\tilde{\phi} \notin \tilde{G}_{\mathcal I}^{\nu}$\,
and then  \,$\tilde{\phi} \notin \tilde{G}_{\mathcal I}$.
Since \,$\tilde{\phi}$\, is any lift of \,$\phi$\,
we conclude that \,$\phi \notin G_{\mathcal{I}}$\, which is a contradiction.
Consequently, we
obtain \,$G_{\mathcal I} \subset {\mathcal I}_{G}$\,
and then \,${\mathcal I}_{G}=G_{\mathcal I}$.
\end{proof}

The next rigidity results are  consequences of Propositions
\ref{pro:bigsma}, \ref{pro:bigsma2} and
Remark \ref{rem:big}.
Next proposition implies Theorem \ref{teo:main6}.

\vskip10pt

\begin{pro}
\label{pro:bigsma3}
Let \,$G$\, be a nilpotent subgroup of \,$\mathrm{Homeo}({\mathbb T}^{2})$.
Then there exists a finite index normal subgroup \,$H$\, of \,$G$\,
$($of index \,$1$ , $2$ , $3$ , $4$ , $6$\, or \,$8)$
containing \,$G_{0}$\,
such that
\[
\big\{ \rho_{\mu}(\tilde{\phi}) \ \ ; \ \ \mu \in {\mathcal P}(\phi) \big\} =
\big\{ \rho_{\nu}(\tilde{\phi})  \ \ ; \ \  \nu \in {\mathcal P}(H) \big\}
\]
for any lift \,$\tilde{\phi}$\, of every \,$\phi \in G_{0}$\,.
In particular we obtain \,${\mathcal I}_{G}={\mathcal I}_{H}=H_{\mathcal I}$\,.
Moreover \,${\mathcal I}_{G}$\,
is a normal subgroup of \,$G$.
\end{pro}


\begin{proof}
If \,$G$\, satisfies Condition $(**)$ then Proposition \ref{pro:bigsma} gives us \,$H=G$. In the
case where \,$G$\, does not satisfy Condition $(**)$ we have from Remark \ref{rem:big} that:
\begin{itemize}
\item
If \,$[G]$\, is finite then we take \,$H=G_{0}$\,;

\item
If \,$[G]=\langle -D \rangle$\, (resp. \,$[G]=\langle D \,,-Id \rangle$)\, then we take \,$H$\, as the inverse image of \,$\langle D^{2} \rangle$\, (resp. \,$\langle D \rangle$) by the map
\,$\psi \in G \mapsto [\psi]\in [G]$.
\end{itemize}
In the latter case \,$H$\, is an index \,$2$\, normal subgroup of \,$G$.
In the former case we just have to notice that
the list of finite nilpotent subgroups of \,$\mathrm{GL}(2,{\mathbb Z})$\, consists of the
subgroups of \,$C_{6}$\, and the subgroups of \,$D_{4}$.
It can be proved by using Remark \ref{rem:spec:01}.
The group $H$ satisfies Condition $(**)$.
We have ${\mathcal I}_{G}={\mathcal I}_{H}$ by definition whereas
${\mathcal I}_{H} = H_{\mathcal I}$ is a consequence of Proposition \ref{pro:bigsma2}.
Since $H_{\mathcal I}$ is a subgroup of $H$, it follows that ${\mathcal I}_{G}$ is a subgroup of $G$.
It is normal in $G$ since ${\mathcal I}_{\mathrm{Homeo}_{0}({\mathbb T}^{2})}$ is a normal subset
of $\mathrm{Homeo}_{0}({\mathbb T}^{2})$.
\end{proof}

\vskip10pt

\begin{rem}
The calculations in Lemma \ref{lem:rinv}
and Propositions \ref{pro:bigsma},
\ref{pro:bigsma2}, \ref{pro:bigsma3} work analogously if
\,$G$\, is a nilpotent subgroup of \,$\mathrm{Homeo}_{0}(S)$\, where
\,$S$\, is a compact orientable surface of genus \,$g \geq 2$.
As a consequence the obvious adaptation of Theorem \ref{teo:main6}
to the case \,$g \geq 2$\, holds true.
\end{rem}

\vskip10pt

\begin{pro}
\label{pro:ano}
Let \,$G$\, be a nilpotent subgroup of \,$\mathrm{Homeo}({\mathbb T}^{2})$\,
such that \,$[G]$\, contains an Anosov class, i.e.
there exists \,$\psi \in G$\, such that
\,$\mathrm{spec}[\psi] \cap {\mathbb S}^{1} = \emptyset$\,.
Then \,${\mathcal I}_{G}$\, is a finite index normal subgroup of \,$G_{0}$\,.
\end{pro}


\begin{proof}
The group \,$[G]$\, is of the form \,$\langle A \rangle$\, or \,$\langle A\,,-Id \rangle$\, where
\,$\mathrm{spec}(A) \cap {\mathbb S}^{1} = \emptyset$\,.
Therefore we obtain \,${\mathcal I}_{G}=G_{\mathcal I}$\, by
Proposition \ref{pro:bigsma2}.
The group \,$G_{\mathcal I}$\, is a finite index normal subgroup of \,$G_{0}$\,
by Proposition \ref{pro:bffi}.
\end{proof}

\begin{example*}
The property \,$G_{\mathcal I}= {\mathcal I}_{G}$\, is false if we do not impose any
condition on the nilpotent group \,$[G]$.
Let us consider a $C^{\infty}$-diffeomorphism \,$\tilde{h}:{\mathbb R}^{2} \to {\mathbb R}^{2}$\,
of the form \,$\tilde{h}(x,y) =(x +f(x),y + g(x))$\, where
\[ f(x+1) \equiv f(x), \ f(x) \equiv -f(-x), \ g(x+1) \equiv g(x) \
\mathrm{and} \ g(x) \equiv -g(-x). \]
We always have \,$f(0)=f(1/2)=g(0)=g(1/2)=0$\,, hence we obtain
\,$\{0,1/2\} \times {\mathbb R} \subset \mathrm{Fix}(\tilde{h})$.
We can choose \,$f,g$\, such that \,$x+f(x)$\, has fixed point set
\,$\{0,1/4,1/2\}$\, at \,$[0,1/2]$\, and \,$g(1/4) \neq 0$.
The diffeomorphism \,$\tilde{h}$\, is a lift of an element
\,$h$\, of \,$\mathrm{Diff}_{0}({\mathbb T}^{2})$.
We denote by \,$\phi$\, the element of \,$\mathrm{Diff}({\mathbb T}^{2})$\,
having the lift \,$\tilde{\phi}(x,y)=(-x,-y)$. We define the groups
\[ \tilde{G} =\langle \,\tilde{h}\,, -Id \, \rangle \ \   ; \ \ G =\langle \, h\,, \phi \,\rangle . \]
They are both abelian by construction, \,$G_{0}= \langle h \rangle$\,
and \,$\tilde{G}$\, is a lift of \,$G$.
The group \,$G_{0}$\, coincides with \,$G_{\mathcal I}$\,
(Proposition 3.4 of \cite{rib02} and Proposition \ref{pro:bffi}).
If \,$h^{k}$\, belongs to \,${\mathcal I}_{G}$\, then \,$\tilde{h}^{k}$\,
is its irrotational lift because \,$\mathrm{Fix}(\tilde{h}^{k}) \neq \emptyset$.
But the Lebesgue measure \,$\mathrm{Leb}$\, in the circle \,$\{1/4\} \times {\mathbb R}/{\mathbb Z}$\,
is $h$-invariant and \,$\rho_{\mathrm{Leb}} (\tilde{h}^{k}) = (0,k g(1/4))$.
We obtain \,${\mathcal I}_{G}=\{Id\}$\, and then \,${\mathcal I}_{G} \neq G_{\mathcal I}$.
\end{example*}


\vskip30pt
\section{Proof of the Theorem \ref{teo:main4}}

In order to show Theorem \ref{teo:main4}
we reduce the general situation to more tractable ones. The main tools are
a Thurston decomposition of all elements in the normal subgroup $H$
\`{a} la Franks, Handel and Parwani \cite{fhp01} and
Theorem \ref{teo:plane}, that allows us to find global fixed points by
lifting the problem to the universal covering.

First we deal with the simple case in which $\mathrm{supp}(\mu)$
admits a $\tilde{\phi}$-invariant compact lift.

\vskip10pt
\begin{pro}
\label{pro:boucas}
Let \,$G= \langle H,\phi \rangle$\, be a nilpotent subgroup of
\,$\mathrm{Diff}_{0}^{1}({\mathbb T}^{2})$\,
where \,$H$\, is a normal subgroup of \,$G$. Suppose  there exist a
lift \,$\tilde{\phi}$\, of \,$\phi$\, to the universal covering and a non-empty compact
 set \,$\tilde{C}\subset\mathbb{R}^{2}$\, which is $\tilde{\phi}$-invariant.
If
\,$\pi(\tilde{C})$\, is contained in \,$\mathrm{Fix}(H)$\,
then \,$G$\, has a global fixed point.
\end{pro}

\begin{proof}
If \,$\tilde{C}$\, is not $\tilde{\phi}$-minimal we replace it by a
$\tilde{\phi}$-minimal set contained in \,$\tilde{C}$.
Let \,$\nu$\, be a $G$-invariant Borel probability measure
on \,$\mathbb{T}^{2}$\, with
\,$\mathrm{supp}(\nu) \subset \pi(\tilde{C})$.
Fix \,$\tilde{p} \in \tilde{C}$\, and let \,$\tilde{H}$\, be the following lift of \,$H$:
$$\tilde{H}:= \big\{ \tilde{h} \in \mathrm{Diff}_{0}^{1}({\mathbb R}^{2}) \ \ ;
\ \ h\in H \quad \text{and} \quad \tilde{h}(\tilde{p})=\tilde{p} \big\}.$$
We claim that
$\tilde{H}$ is a normal subgroup of $\langle \tilde{H},\tilde{\phi} \rangle$.
For this let us denote $H^{j}=Z^{(j)}(G) \cap H$. We will show by induction on
\,$j\geq0$\, that the lift \,$\tilde{H}^{j} \subset \tilde{H}$\, of \,$H^{j}$\,
is a normal subgroup of \,$\langle \tilde{H},\tilde{\phi} \rangle$.

The result is obviously true for \,$j=0$.

Now, suppose the result holds true for some integer \,$j\geq0$\,. Given
\,$\ell \in \mathbb{Z}$\, we have that the elements of
\,$\tilde{\phi}^{l} \, \tilde{H}^{j} \, \tilde{\phi}^{-l}$\, fix \,$\tilde{\phi}^{l}(\tilde{p})$\,. From the induction hypothesis we have that
\,$\tilde{\phi}^{l} \, \tilde{H}^{j} \, \tilde{\phi}^{-l}=\tilde{H}^{j}$. Thus
the $\tilde{\phi}$-orbit of \,$\tilde{p}$\, is contained in
\,$\mathrm{Fix}(\tilde{H}^{j})$\, and then \,$\tilde{C} \subset \mathrm{Fix}(\tilde{H}^{j})$.
Consequently we conclude that
\,$\pi^{-1}(\pi(\tilde{C})) \subset \mathrm{Fix}(\tilde{H}^{j})$\, since the elements of
\,$\tilde{H}^{j}$\, commute with the covering transformations.

Given \,$h \in H^{j+1}$\, and \,$l \in {\mathbb Z}$\, we know that
\,$h':=[\,\phi^{l}\,,h\,]$\, belongs to  \,$H^{j}$. On the other hand,
\,${\rho_{\nu}\big|}_{\langle \tilde{H},\tilde{\phi} \rangle}$\,
is a morphism of groups by  Lemma \ref{lem:mor}  and
we obtain \,$\rho_{\nu}\big(\big[\, \tilde{\phi}^{l} \,,\tilde{h}\,\big]\big)=(0\,,0)$.
Moreover, \,$\tilde{h}' \in \tilde{H}^{j}$\, also satisfies
\,$\rho_{\nu} \big(\tilde{h}'\big)=(0\,,0)$\, since
\,$\pi^{-1}(\pi(\tilde{C})) \subset \mathrm{Fix}(\tilde{H}^{j})$\,
and \,$\text{supp}(\nu) \subset \pi(\tilde{C})$.
Hence we obtain \,$\big[\,\tilde{\phi}^{l} \,,\tilde{h}\,\big]=\tilde{h}'$\, and
\[ \tilde{\phi}^{l} \circ \tilde{h} \circ \tilde{\phi}^{-l}   = \tilde{h}' \circ \tilde{h} \in \tilde{H}^{j+1} \]
proving that \,$\tilde{H}^{j+1}$\, is a normal subgroup of \,$\langle \tilde{H},\tilde{\phi} \rangle$.
Since \,$G$\, is nilpotent we conclude that \,$\tilde{H}$\, is normal in
\,$\langle \tilde{H},\tilde{\phi} \rangle$.

Now, in particular we have that
\,$\tilde{C}$\, is contained in \,$\mathrm{Fix}(\tilde{H})$\, and it is
 is a non-empty compact set
which is invariant by \,$\langle \tilde{H},\tilde{\phi} \rangle$.
 Moreover, from \,$\tilde{\phi}\,\tilde{H}\,\tilde{\phi}^{-1}=\tilde{H}$\, and
\,${\phi}\,{H}\,{\phi}^{-1}={H}$\, we conclude that
\,$[\langle\tilde{H},\tilde{\phi}\rangle,\langle\tilde{H},\tilde{\phi}\rangle]\subset \tilde{H}$\,
and
\,$[\langle {H}, {\phi}\rangle,\langle {H}, {\phi}\rangle]\subset  {H}$.
Consequently, the natural projection from
\,$[\langle\tilde{H},\tilde{\phi}\rangle,\langle\tilde{H},\tilde{\phi}\rangle]$\, to
\,$[\langle {H}, {\phi}\rangle,\langle {H}, {\phi}\rangle]$\, is an isomorphism proving that
\,$\langle \tilde{H},\tilde{\phi} \rangle$\, is also a nilpotent group.

%

Thus if \,$G$\, is finitely generated then
\,$\mathrm{Fix} \big(\langle \tilde{H},\tilde{\phi} \rangle \big) \neq\emptyset$\,  by
Theorem \ref{teo:plane} and
\,$G$\, has a global fixed point. The general case for \,$G$\, follows now from the  finite intersection
property on \,$\mathbb{T}^{2}$.
\end{proof}

Let us return to the statement of Theorem \ref{teo:main4} and consider
the following special case\,:   \,$\mathrm{supp}(\mu)$\, is a finite set.
Let \,$p\in \text{supp}(\mu)$\, and \,$\tilde{p}\in \pi^{-1}(p)$.
The ergodic nature of \,$\mu$\,
guarantees that the \,$\phi$-orbit of \,$p$\, coincides with \,$\mathrm{supp}(\mu)$\,.
Moreover, the property \,$\rho_{\mu}\big(\tilde{\phi}\big)=(0\,,0)$\,
implies that the $\tilde{\phi}$-orbit of \,$\tilde{p}$\, is finite.
Consequently,
\,$\mathrm{supp}(\mu)$\, has a $\tilde{\phi}$-invariant compact lift.
In particular we obtain\,:

\vskip10pt
\begin{cor}
\label{cor:boucas}
Consider the hypotheses of Theorem \ref{teo:main4}.
If \,$\mathrm{supp}(\mu)$\, is a finite set
then \,$G$\, has a global fixed point.
\end{cor}


Later on we consider the Thurston decomposition of the normal subgroup $H$.
The next lemma is used to rule out the existence of non-essential
reducing curves. Moreover, let us remark that if \,$\text{supp}(\mu)=\mathbb{T}^{2}$\,
in Theorem \ref{teo:main4} then \,$H=\{Id\}$\,
and the result follows from the referred theorem of Franks in \cite{f22}. Then, in what
follows we are assuming that \,$\text{supp}(\mu)\neq \mathbb{T}^{2}$.

\vskip10pt

\begin{lem}
\label{lem:noness}
Consider the setting in Theorem \ref{teo:main4}.
Suppose  there exists a
non-essential simple closed curve \,$\gamma$\,
such that$\,:$
\begin{itemize}
\item
$\gamma \subset {\mathbb T}^{2} \setminus \mathrm{supp}(\mu) \,;$

\item
the compact disc \,$D$\, enclosed by \,$\gamma$\, contains points of
\,$\mathrm{supp}(\mu) \,;$

\item
$\phi^{k}(\gamma)$\, is isotopic to \,$\gamma$\,
in \,$\mathbb{T}^{2}\setminus\mathrm{supp}(\mu)$\,  for some
\,$k \in {\mathbb Z}^{+}$.

\end{itemize}
Then \,$G$\, has a global fixed point.
\end{lem}

\begin{proof}
Let us denote \,$C= D \cap \mathrm{supp}(\mu)$\, and let
\,$\tilde{D}$\, be a lift  of
\,$D$\, in \,${\mathbb R}^{2}$.
Let \,$\tilde{\phi}$\, be the lift of \,$\phi$\, such that
\,$\rho_{\mu}\big(\tilde{\phi}\big)=(0\,,0)$\, and let us fix the lift
\,$\tilde{C} = \pi^{-1}(C) \cap \tilde{D}$\,  of \,$C$.

Since \,$\phi^{k}(\gamma)$\, is isotopic to \,$\gamma$\,
in \,$\mathbb{T}^{2}\setminus\mathrm{supp}(\mu)$\, we obtain  \,$\phi^{k}(C) = C$. Consequently, the compact set
\,$\cup_{j=0}^{k-1} \, \phi^{j}(C)$\, is
\,$\phi$-invariant and has positive measure.
Now, it follows from the $\phi$-ergodicity of \,$\mu$\, that
$$\text{supp}(\mu) = \bigcup_{j=0}^{k-1} \phi^{j}(C).$$

We have \,$\tilde{\phi}^{\,k}\big(\tilde{C}\big)
= \tilde{C} + (m\,,n)$\,
for some \,$(m\,,n)\in \mathbb{Z}^{2}$.
The property \,$\rho_{\mu}(\tilde{\phi})=(0\,,0)$\, implies
\,$(m\,,n)=(0\,,0)$\, and we obtain
\,$\tilde{\phi}^{\,k}(\tilde{C})=\tilde{C}$.
Thus \,$\cup_{j=0}^{k-1} \,\tilde{\phi}^{\,j}\big(\tilde{C}\big)$\,
is a non-empty compact \,$\tilde{\phi}$-invariant set whose projection by \,$\pi$\,
 is contained in \,$\text{Fix}(H)$\,.
Then the group \,$G$\, has a global fixed point by Proposition \ref{pro:boucas}.
\end{proof}

A common trait in several subcases in the proof of Theorem \ref{teo:main4}
is that if \,$\phi$\, preserves a non-trivial  simple closed curve
modulo isotopy relative to \,$\mathrm{supp}(\mu)$\, then \,$G$\, has a global fixed point.
In particular such phenomenon forces the Thurston decomposition of \,$H$\,
to be trivial.

\vskip10pt

\begin{lem}
\label{lem:nness}
Consider the hypotheses of Theorem \ref{teo:main4}.
Then either \,$\mathrm{Fix}(G) \neq \emptyset$\, or
any element of \,$H$\, is isotopic to the identity
relative to \,$\mathrm{supp}(\mu)$.
\end{lem}


\begin{proof}
Suppose \,$\mathrm{Fix}(G) = \emptyset$. Let us prove that
any element of \,$H$\, is isotopic to the identity
relative to \,$\mathrm{supp}(\mu)$.

Since \,$\text{Fix}(G)=\emptyset$\, there exists a finitely
generated subgroup \,$J$\, of \,$G$\, containing \,$\phi$\, and such that
\,$\mathrm{Fix}(J)=\emptyset$. Now it suffices to prove that
\,$h$\, is isotopic to the identity relative to \,$\mathrm{supp}(\mu)$\,
for every finitely generated subgroup \,$L$\, of \,$G$\, containing \,$J$\,
and \,$h \in L \cap H$. Hence up to replace \,$G$\, with \,$L$\, and
\,$H$\, with \,$L \cap H$\, we can suppose that \,$G$\, is finitely generated
and \,$\mathrm{Fix}(G) = \emptyset$.

Let us denote \,$K = \mathrm{supp}(\mu)$. Since \,$K$\, is a subset of
\,$\mathrm{Fix}(H)$\, then
\,$K$\, is $G$-invariant.
Now we apply Proposition \ref{pro:tnf}.
We can suppose that the Thurston decomposition is minimal, i.e.
that it has a minimal number of reducing curves.
It suffices to show
\,${\mathcal R}=\emptyset$\, and \,$\sharp K=\infty$.
The former property implies that \,${\mathbb T}^{2}$\, is the unique connected component of
\,${\mathbb T}^{2} \setminus {\mathcal A}$\, whereas the latter property guarantees that
\,$h$\, is isotopic to the identity relative to \,$K$\, for any \,$h \in H$.

Moreover since ${\mathcal R}$ is $\phi$-invariant modulo isotopy relative to
$K$, there exists $k \in {\mathbb Z}$ such that $\phi^{k}(\gamma)$ is isotopic to
$\gamma$ in ${\mathbb T}^{2} \setminus K$ for any connected component $\gamma$ of ${\mathcal R}$.
Hence we may admit
that \,${\mathcal R}$\, does not contain non-essential reducing curves by
Proposition \ref{pro:tnf} and Lemma \ref{lem:noness}.

Suppose \,${\mathcal R} \neq \emptyset$.
Since the curves in \,${\mathcal R}$\, are simple and pairwise disjoint, they
are free homotopic to an essential curve \,$\kappa$\, in \,${\mathbb T}^{2}$.
Consider a connected component \,$U$\, of \,${\mathbb T}^{2} \setminus {\mathcal R}$.
Then either \,$\sharp (U \cap K) = \infty$\, or
\,$\sharp (U \cap K) < \infty$\, and \,$\chi (U \setminus K)<0$\, by the properties
of the Thurston decomposition.
Since $U$ is an annulus ($\chi (U)=0$), it
contains points of \,$K$.
The curve $\kappa$ defines a non-trivial class $(a,b)$ in
$H^{1}({\mathbb T}^{2}, {\mathbb Z}) \sim {\mathbb Z}^{2}$. The property
\,$\rho_{\mu}(\tilde{\phi})=(0\,,0)$\, implies
\,$\phi(U \cap K)=U \cap K$,
otherwise we obtain
\[ \liminf_{n \to \infty}
\left\| \left< \frac{\tilde{\phi}^{n}(\tilde{z})- \tilde{z}}{n}, (-b,a) \right> \right\| >0 \]
for all $z \in K \cap U$ where $<,>$ is the canonical inner  product.
The last property contradicts the Birkhoff's ergodic theorem.

Since \,$\phi(U \cap K)=U \cap K$\, and
\,$\mu$\, is $\phi$-ergodic, we obtain that
\,${\mathbb T}^{2} \setminus {\mathcal R}$\, has one connected component.
We can suppose that \,$K$\, is an infinite set, otherwise
\,$\mathrm{Fix}(G)$\, is non-empty by
Corollary \ref{cor:boucas}.

If \,${\mathcal R}$\, contains exactly one reducing curve then
any \,$h$\, in \,$H$\, is isotopic to a Dehn twist \,$\theta$\, relative to
\,$K$. The diffeomorphism \,$h$\, is isotopic to the identity;
thus the Dehn twist is trivial.
We deduce that \,$h$\, is isotopic to the identity relative to \,$K$.
If we remove the unique curve in \,${\mathcal R}$\, we still have
a Thurston decomposition, contradicting the minimality of
${\mathcal R}$.
\end{proof}

Lemma \ref{lem:nness} implies that we can consider a canonical lift of \,$H$,
namely the identity lift \,$\tilde{H}$\, of \,$H$\,
to the universal covering of  \,${\mathbb T}^{2} \setminus \mathrm{supp}(\mu)$.
The practical consequence is that from now on we focus on properties of \,$\phi$\,
and the role of \,$H$\, is secondary.

\vskip10pt

\begin{defi}
\label{def:q2}
Let \,$\psi \in \mathrm{Homeo}({\mathbb T}^{2})$\, and let \,$K \subset \mathbb{T}^{2}$\, be
a non-empty  $\psi$-invariant compact set.
Let us fix \,$z \in \mathrm{Fix}(\psi) \setminus K$\,
and let
\,$U$\, be the connected component of \,${\mathbb T}^{2} \setminus K$\,
containing \,$z$. Fixed a lift \,$\tilde{z}$\, of \,$z$\,
in the universal covering \,$\tilde{U}$\, of \,$U$\, we will
denote by \,$\tilde{\psi}_{\tilde{z}}$\, the lift of \,$\psi$\, to \,$\tilde{U}$\, such that
\,$\tilde{\psi}_{\tilde{z}}(\tilde{z})=\tilde{z}$.
\end{defi}


Our goal is finding a lift \,$\tilde{\psi}_{\tilde{z}}$\, having a
compact fixed
point set and then applying Theorem \ref{cor:plane}.
The next definition is useful to study the obstruction to the existence
of such a lift.

\vskip10pt

\begin{defi}
\label{def:q}
Let us define a subset \,$Q(K, \psi)$\, of
\,$\mathrm{Fix}(\psi)\setminus K$.
Given \,$z \in \mathrm{Fix}(\psi)\setminus K$\, we say that
\,$z \not \in Q(K, \psi)$\, if
there exists a simple closed curve
\,$\varsigma:[\,0\,,1\,] \to {\mathbb T}^{2} \setminus K$\, such that
\,$\varsigma(0)=\varsigma(1)=z$\,,
\,$\varsigma$\, is essential in \,${\mathbb T}^{2}$\, and
\,$[\psi \circ \varsigma] = [\varsigma]$\, where \,$[\varsigma]$\,
denotes the
homotopy class defined by \,$\varsigma$\, in \,$\pi_{1}({\mathbb T}^{2} \setminus K, z)$.

\end{defi}


The property \,$Q(\mathrm{supp}(\mu),\phi) \neq \emptyset$\, is a necessary condition
for the existence of a lift of \,$\phi$\, with a non-empty compact fixed point set.
Curiously, it is also a sufficient condition.

\vskip10pt

\begin{lem}
\label{lem:compact}
Consider the setting in Theorem \ref{teo:main4}. Suppose there exists
\,$z \in Q(\mathrm{supp}(\mu),\phi)$\, and let
 \,$U$\, be the connected component of \,${\mathbb T}^{2} \setminus \mathrm{supp}(\mu)$\,
containing \,$z$.
Then either \,$\mathrm{Fix}(G) \neq \emptyset$\, or
\,$\mathrm{Fix}\big(\tilde{\phi}_{\tilde{z}}\big)$\, is a non-empty compact subset of \,$\tilde{U}$.
\end{lem}


\begin{proof}
Suppose $\mathrm{Fix}(G) = \emptyset$\, and let us
denote \,$\tilde{\phi}=\tilde{\phi}_{\tilde{z}}$\,  following Definition
\ref{def:q2}.
Let \,$\pi': \tilde{U} \to U$\, be the universal covering map.
By construction \,$\tilde{z}$\, belongs to \,$\mathrm{Fix}\big(\tilde{\phi}\big)$.
 Moreover,
it is clear that \,$\mathrm{Fix}(\phi) \cap U$\, is
a compact set since \,$\mathrm{Fix}(G) = \emptyset$\, implies
 \,$\mathrm{Fix}(\phi)\cap\mathrm{supp}(\mu)=\emptyset$. Hence it
suffices to prove that the map
\,$\pi'\big|_{\mathrm{Fix}(\tilde{\phi})}:\mathrm{Fix}\big(\tilde{\phi}\big) \to
\mathrm{Fix}(\phi)$\, is injective. We will prove this by contradiction.

 Suppose that \,$\pi'\big|_{\mathrm{Fix}(\tilde{\phi})}$\,
is not injective.  Then
\,$\tilde{\phi}$\, commutes with some covering transformation and
there
exists a non-trivial homotopy class \,$[\beta] \in \pi_{1}(U, z)$\, such that
\,$[\phi \circ \beta] = [\beta]$.
Since \,$[\phi \circ \beta] = [\beta]$\, it follows from Lemma 2.12 of \cite{fhp01} that, up to replace \,$\beta$\, with another
representative of \,$[\beta]$\,
there exists a finite type compact and connected subsurface \,$T$\, contained in \,$U$,
containing \,$\beta$\,, with no connected component of \,$\partial T$\,
bounding a disc in \,${\mathbb T}^{2} \setminus \text{supp}(\mu)$\, and a homeomorphism
\,$\theta: {\mathbb T}^{2} \to {\mathbb T}^{2}$\, such that
\,$\theta$\, is isotopic to  \,$\phi$\, relative to
\,$\text{supp}(\mu)$,
\,$\theta(T) =T$\, and \,$\theta\big|_{T}$\, has finite order.

If  \,$\partial T$\, has a connected component that
bounds a compact disk \,$D$\, in \,${\mathbb T}^{2}$\, then \,$G$\, has a
global fixed point by Lemma \ref{lem:noness},  contradicting the condition
\,$\text{Fix}(G) = \emptyset$\,.
So we can suppose that all the boundary curves in \,$\partial T$\,
are essential in \,${\mathbb T}^{2}$. Since they do not intersect each other
they are all homotopic in \,${\mathbb T}^{2}$\, to some essential
simple closed curve \,$\kappa'$\, and the surface \,$T$\, is an annulus.
The class \,$[\beta]$\, is not null-homotopic
in \,$U$\, and
hence it is not null-homotopic in the annulus. Consequently,
\,$\beta$\, is homotopic to a non-trivial multiple of the
core fundamental class \,$\varsigma$\, of the annulus
and \,$[\phi\circ \varsigma] = [\varsigma]$\,
in \,$\pi_{1}(U, z)$. But this also leads to a contradiction
since \,$z \in Q(\mathrm{supp}(\mu),\phi)$\, and the proof is complete.
\end{proof}

\vskip10pt

\begin{pro}\label{prop:main}
Consider the setting in Theorem \ref{teo:main4}.
If \,$Q(\mathrm{supp}(\mu),\phi)$\, is non-empty then
\,$G$\, has a global fixed point.
\end{pro}


\begin{proof}
Let us denote \,$K=\mathrm{supp}(\mu)$. If \,$K \cap \mathrm{Fix}(\phi) \neq \emptyset$\,
then there is nothing to prove.
Now let  \,$z \in Q(K, \phi)$\, and suppose
\,$K \cap \mathrm{Fix}(\phi) = \emptyset$. Consider the
connected component \,$U$\, of \,${\mathbb T}^{2} \setminus K$\,
containing \,$z$. An orientation-preserving homeomorphism \,$\psi$\, defined in a
connected topological manifold \,$M$\, leaves the connected
components of \,$M \setminus \mathrm{Fix} (\psi)$\, invariant \cite{Brown-Kister}.
Since \,$K \subset \mathrm{Fix} (H)$\, we deduce that \,$U$\, is $H$-invariant.

If \,$U$\, is a topological disc then the
restriction of \,$G$\, to \,$U$\, defines
a nilpotent group such that \,$\mathrm{Fix}\big(\phi\big|_{U}\big)$\, is
a non-empty  compact set and we  obtain \,$\mathrm{Fix}(G) \neq \emptyset$\, by
Theorem \ref{cor:plane}.

Suppose that \,$U$\, is an annulus and let \,$\kappa'$\, be a core curve of the
annulus. The curve \,$\kappa'$\, is non-essential in \,${\mathbb T}^{2}$,
otherwise \,$z$\, does not belong to \,$Q(K, \phi)$.
Thus we obtain \,$\mathrm{Fix}(G) \neq \emptyset$\,  by Lemma \ref{lem:noness}.

Hence we can suppose from now on
that \,$U$\, has negative Euler characteristic. Let us
denote \,$\tilde{\phi}=\tilde{\phi}_{\tilde{z}}$\,
(cf. Definition \ref{def:q2}).
Following Lemma \ref{lem:compact} we can assume
that
\,$\mathrm{Fix}\big(\tilde{\phi}\big)$\, is a non-empty compact set.
Moreover we can admit that \,$h$\, is isotopic to the identity
relative to \,$K$\,
for any \,$h \in H$\, by Lemma \ref{lem:nness}.

Now, let \,$\tilde{H} \subset \mathrm{Diff}_{0}^{1}(\tilde{U})$\,
be the subgroup composed by the
identity lifts of elements of \,$H$. The identity lift is well-defined since \,$\chi (U)<0$.
Given that the elements of \,$\tilde{H}$\, are the identity lifts of elements in
\,$H$\, we have \,$\tilde{\phi} \, \tilde{H} \, \tilde{\phi}^{-1} = \tilde{H}$.
Consequently, \,$\tilde{H}$\, is a normal subgroup of  \,$\langle \tilde{H},\tilde{\phi} \rangle$.
Moreover, from \,$[\langle \tilde{H},\tilde{\phi} \rangle,\langle \tilde{H},\tilde{\phi} \rangle]\subset \tilde{H}$\,
we conclude that \,$\langle \tilde{H},\tilde{\phi} \rangle$\, is a nilpotent group
as we have done at the end of the proof of Proposition \ref{pro:boucas}.
Therefore \,$\langle \tilde{H},\tilde{\phi} \rangle$\, has a global fixed point by
Theorem \ref{cor:plane}  and we obtain
\,$\mathrm{Fix}(G) \neq \emptyset$.
\end{proof}

In order to complete the proof of Theorem \ref{teo:main4} it suffices to show
that \,$Q(\mathrm{supp}(\mu), \phi)$\, is non-empty.
The next definition is useful to work in the universal covering.

\vskip10pt

\begin{defi}
\label{def:qx}
Let \,$\mathfrak{f} \in \mathrm{Homeo}_{+}({\mathbb R}^{2})$.
Consider a   $\mathfrak{f}$-invariant compact set \,$\mathcal{O}$\, such that
\,$\mathrm{Fix}(\,\mathfrak{f}\,) \, \cap \, \mathcal{O} = \emptyset$.
Let us define a subset \,$P(\mathcal{O},\mathfrak{f}\,)$\, of \,$\mathrm{Fix}(\,\mathfrak{f}\,)$.
We say that \,$\mathfrak{p} \in \mathrm{Fix}(\,\mathfrak{f}\,)$\, does not belong to
\,$P(\mathcal{O},\mathfrak{f}\,)$\, if there exists a
path  \,$\gamma:[\,0\,,\infty) \to {\mathbb R}^{2}\setminus \mathcal{O}$\,
such that \,$\gamma(0)= \mathfrak{p}$\,,
\,$\lim_{t \to \infty} \|\gamma(t)\| =\infty$\,
and \,$\mathfrak{f} \circ \gamma$\, is properly homotopic to \,$\gamma$\,
in  \,${\mathbb R}^{2} \setminus \mathcal{O}$\, relative to \,$0$.
\end{defi}


\vskip10pt

In \cite[Proposition 5.3]{fhp01} Franks-Handel-Parwani prove the following result.

\vskip10pt

\begin{teo}
\label{teo:fhpg}
 Consider the setting in the above definition. If
  \,$\mathcal{O} \neq \emptyset$\, then
  \,$P(\mathcal{O},\mathfrak{f}\,) \neq \emptyset$.
\end{teo}


 The next lemma completes the proof of Theorem \ref{teo:main4}. In its proof we
are going to use the above theorem only when the compact set is
a finite orbit. In such a case the result is due to Gambaudo \cite{Gamba}.

\vskip10pt

\begin{lem}\label{lem:main}
Let \,$\phi \in \mathrm{Diff}_{0}^{1}({\mathbb T}^{2})$.
Suppose  there exist a $\phi$-invariant ergodic probability measure
\,$\mu$\, and a lift  \,$\tilde{\phi} \in \mathrm{Diff}^{1}_{0}({\mathbb R}^{2})$\,
such that \,$\rho_{\mu}\big(\tilde{\phi}\big)=(0\,,0)$. If
\,$\mathrm{Fix}(\phi) \cap \mathrm{supp}(\mu)=\emptyset$\, then
\,$Q(\mathrm{supp}(\mu), \phi) \neq \emptyset$.
\end{lem}


\begin{proof}
Since \,${\rho}_{\mu}\big(\tilde{\phi}\big)=(0\,,0)$\,
and \,$\mu$\, is $\phi$-ergodic
we have that \,$\text{Fix}(\phi)\neq\emptyset$\, by Franks Theorem \ref{teo:franks}.
Denote \,$K=\mathrm{supp}(\mu)$\, and let \,$p$\, be a fixed point of \,$\phi$.
If \,$p$\, belongs to \,$Q(K,\phi)$\,
then there
is nothing to prove.

Suppose \,$p \notin Q(K,\phi)$\, and consider the simple closed curve
\,$\varsigma_{0}$\, in \,${\mathbb T}^{2} \setminus K$\, provided by Definition \ref{def:q}. This loop is based on \,$p \in \text{Fix}(\phi)$\, and it is essential in \,$\mathbb{T}^{2}$. Moreover,  \,$\varsigma_{0}$\, and
\,$\phi \circ \varsigma_{0}$\, are homotopic in \,$\mathbb{T}^{2} \setminus K$\, relative to the base point \,$p$.
Let \,$v \in {\mathbb Z}^{2}$\, be the nonzero vector such that
the translation \,$T_{v}$\, is the covering transformation
for \,$\pi:\mathbb{R}^{2}\rightarrow \mathbb{T}^{2}$\,  associated to \,$\varsigma_{0}$.
A result by Koropecki and Tal (Lemma 4.4 of \cite{koro01}),
 inspired in Atkinson's lemma \cite{Atk},
implies that there exists a total measure set
\,$E \subset K \subset {\mathbb T}^{2}$\,
such that for any
\,$z \in E$\, there exists an increasing sequence \,$(n_{k})_{k\geq1}$\, of
positive integers
such that
\begin{align}\label{KT:01}
 \phi^{n_{k}}(z) \to z \quad \mathrm{and} \quad
<\tilde{\phi}^{n_{k}}(\tilde{z}) - \tilde{z} \, , v > \, \to 0
\end{align}
where \,$<\,,>$\, is the canonical inner product and \,$\tilde{z}$\, is any lift of \,$z$.
Moreover, since \,$\varsigma_{0}$\,  is essential in \,$\mathbb{T}^{2}$\,  and
\,$[\phi\circ\varsigma_{0}]=[\varsigma_{0}]$\, in
\,$\pi_{1}({\mathbb T}^{2} \setminus K,p)$\, we conclude that \,$\tilde{z}$\, and
\,$\tilde{\phi}(\tilde{z})$\, are in the same  connected component of
\,${\pi}^{-1}\big(\mathbb{T}^{2} \setminus \varsigma_{0}\big([\,0\,,1\,]\big)\big)$\, for any
\,$\tilde{z} \in {\pi}^{-1}(K)$\,.
In particular we obtain from \eqref{KT:01} that
\,$\tilde{\phi}^{n_{k}}(\tilde{z}) \to \tilde{z}$\, whenever \,$\pi(\tilde{z})\in E$.
Therefore
the set of points in \,$K$\, whose preimages are
$\tilde{\phi}$-recurrent has total measure.

We will prove the lemma by using Theorem \ref{teo:fhpg}. For this we consider a local small $C^{0}$-\,perturbation of \,$\tilde{\phi}$\, without changing the set of fixed points and getting  a new non-empty compact set which is invariant by the perturbation. This new invariant compact set for the perturbation  will be a periodic orbit contained in \,$\pi^{-1}(K)$\,.

Let us fix \,$\tilde{q} \in \pi^{-1}(K)$\,  a $\omega$-recurrent point for \,$\tilde{\phi}$.
We will choose \,$H\in\text{Homeo}_{0}(\mathbb{R}^{2})$\, whose support is
contained in a small compact disk  \,$\tilde{D}$\, containing \,$\tilde{q}$\, in its interior.
We are always assuming that \,$\tilde{D}$\, is a lift of a small disk \,$D$\, contained in \,$\mathbb{T}^{2}$.
Furthermore \,$H$\, is such that\,:

\begin{itemize}
\item
the orbit \,$\mathcal{O}$\, of \,$\tilde{q}$\, by  \,$H \circ \tilde{\phi}$\,
is finite and \,$\mathcal{O} \subset \pi^{-1}(K)$\,;

\item
$\text{Fix}\big(\tilde{\phi}\big) = \text{Fix}\big(H \circ \tilde{\phi}\big)$\,.

\end{itemize}

\noindent
Of course the choice of \,$H$\, depends on the compact disk \,$D$\, which can be made
arbitrarily small.

 Now, let us remind the reader that Franks proves the existence of a fixed point for
\,$\phi$\,  proving that \,$\tilde{\phi}$\, has a fixed point
(cf. beginning of the proof of Theorem 3.5 in \cite{f22}).
In the same spirit,
to finish the proof of the lemma it suffices to conclude that
\begin{align}\label{con:cont}
\mathrm{Fix}\big(\tilde{\phi}\big) \setminus \pi^{-1}\big(Q(K, \phi)\big) \ \subset \
\mathrm{Fix}\big(H \circ \tilde{\phi}\big) \setminus P({\mathcal O}, H \circ \tilde{\phi}).
\end{align}
From the above relation   we obtain
\,$Q(K, \phi) \neq\emptyset$\, as required by the lemma since we have
\,$P({\mathcal O}, H \circ \tilde{\phi})\neq\emptyset$\, by Theorem \ref{teo:fhpg}.

To prove \eqref{con:cont} for a suitable choice
of \,$H$\, and \,$D$\, we do the following.  Consider
\,$\tilde{z}_{1} \in \mathrm{Fix}\big(\tilde{\phi}\big) \setminus \pi^{-1}\big(Q(K, \phi)\big)$.
 The point \,$z_{1}=\pi(\tilde{z}_{1})$\, is in
 \,$\text{Fix}(\phi) \setminus Q(K, \phi)$. Consider a simple closed curve \,$\varsigma$\, in
 \,$\mathbb{T}^{2}\setminus K$\, given by Definition \ref{def:q}. We remind that
 \,$\varsigma$\, is based on \,$z_{1}$\,, is essential in \,$\mathbb{T}^{2}$\, and
 \,$[\phi \circ \varsigma]=[\varsigma]$\, in \,$\pi_{1}(\mathbb{T}^{2}\setminus K,z_{1})$\,.
Let \,$\tilde{\varsigma}:[\,0\,,1\,] \rightarrow \mathbb{R}^{2}\setminus \pi^{-1}(K)$\, be the
lift of \,$\varsigma$\, with \,$\tilde{\varsigma}(0)=\tilde{z_{1}}$\, and let \,$T_{w}$\, be the
translation given by the covering transformation for
\,$\pi:\mathbb{R}^{2} \rightarrow \mathbb{T}^{2}$\,
associated to \,$\varsigma$.

We extend \,$\tilde{\varsigma}$\, to \,$[\,0\,,\infty)$\, by defining
\,$\tilde{\varsigma}(t) = T_{w}^{[t]}\big(\tilde{\varsigma}(t-[t])\big)$\, where
\,$[t]$\, denotes the greatest integer less than or equal to \,$t$. It is a continuous path
such that \,$\|\tilde{\varsigma}(t)\| \rightarrow \infty$\,
when \,$t$\, goes to infinity. The loops \,$\varsigma$\, and \,$\phi \circ \varsigma$\,
are homotopic in \,$\mathbb{T}^{2} \setminus K$\,  relative to the base point \,$z_{1}$. Thus the paths
\,$\tilde{\varsigma}\big|_{[0,1]}$\, and
\,$\big(\tilde{\phi} \circ \tilde{\varsigma}\big)\hskip-1pt\big|_{[0,1]}$\, are homotopic relative to the ends \,$\tilde{z}_{1} \,, T_{w}(\tilde{z}_{1})$\, by a homotopy
\,$F: [\,0\,,1\,] \times [\,0\,,1\,] \to {\mathbb R}^{2} \setminus \pi^{-1}(K)$\,. We extend \,$F$\,
to \,$[\,0\,,\infty) \times [\,0\,,1\,]$\, by defining
\,$F(t,s)=T_{w}^{[t]}\big(F(t-[t],s)\big)$. In this way the extended paths \,$\tilde{\varsigma}$\,
and \,$\tilde{\phi} \circ \tilde{\varsigma}$\, are properly homotopic in
\,$\mathbb{R}^{2} \setminus \pi^{-1}(K)$\, relative to \,$0\in [\,0\,,\infty)$\,.

 Furthermore, since the loops \,$\varsigma$\, and \,$\phi \circ \varsigma$\, are homotopic
in \,$\mathbb{T}^{2} \setminus K$\, we conclude that the extended homotopy
\,$F$\,  betweeen the extended curves \,$\tilde{\varsigma}$\, and \,$\tilde{\phi} \circ \tilde{\varsigma}$\, does not intersect
the  set \,$\pi^{-1}\big(\pi(\tilde{D})\big)$\,
 if \,$D$\, is sufficiently small. Thus
\,$\tilde{\phi} \circ T_{u} \circ \tilde{\varsigma} =
\big(H \circ \tilde{\phi}\big)\circ T_{u} \circ \tilde{\varsigma}$\,
and the extended homotopy \,$T_{u} \circ F$\, is a proper homotopy in \,$\mathbb{R}^{2} \setminus \pi^{-1}(K)$\,
between the extended curves \,$T_{u} \circ \tilde{\varsigma}$\, and
\,$\big(H \circ \tilde{\phi}\big)\circ T_{u} \circ \tilde{\varsigma}$\, for every \,$u\in\mathbb{Z}^{2}$.
Since \,$\text{Fix}\big(\tilde{\phi}\big)=\text{Fix}\big(H \circ \tilde{\phi}\big)$\,
by construction we obtain
$$\pi^{-1}(z_{1}) \subset \text{Fix}\big(H \circ \tilde{\phi}\big) \setminus
P(\mathcal{O},H \circ \tilde{\phi})$$
following the Definition \ref{def:qx}.

We remind the reader that \,$\tilde{D}$\,
is the lift containing \,$\tilde{q}$\, of a convenient small disk in
\,$\mathbb{T}^{2}$ containing \,$q=\pi(\tilde{q})$\, in its interior.

Now, we need to make a choice
of \,$D$\, that does not depend on
\,$z\in \text{Fix}(\phi) \setminus Q(K,\phi)$\,. For this we remark that if the above construction holds true
for some  \,$z\in \text{Fix}(\phi) \setminus Q(K,\phi)$\,  then the same is true for all the  points
of \,$\text{Fix}(\phi) \setminus Q(K,\phi)$\, in a small neighborhood
of \,$z$ even without changing the compact disk \,$D$. Moreover we have that
\,$\text{Fix}(\phi) \setminus Q(K,\phi)$\, is a compact set since \,$\text{Fix}(\phi) \cap K=\emptyset$.
The existence of a common choice of \,$D$\, for any
\,$z\in \text{Fix}(\phi) \setminus Q(K,\phi)$\, is a consequence of the compactness of
\,$\text{Fix}(\phi) \setminus Q(K,\phi)$\,, completing the proof of the lemma.
\end{proof}


 Now we prove the Theorems \ref{teo:main4} , \ref{teo:main2} and \ref{teo:main}.

\vskip10pt
\subsection*{Proof of Theorem \ref{teo:main4}}

If \,$\text{Fix}(\phi) \cap \text{supp}(\mu)\neq\emptyset$\, then \,$G$\, has a global fixed point since \,$\text{supp}(\mu) \subset \text{Fix}(H)$\, and \,$G=\langle H,\phi \rangle$\,.
If  \,$\text{Fix}(\phi) \cap \text{supp}(\mu)=\emptyset$\, then Theorem \ref{teo:main4} follows from
Lemma \ref{lem:main} and Proposition \ref{prop:main}.

\vskip10pt
\subsection*{Proof of Theorem \ref{teo:main2}}

We can suppose that \,$G$\, is finitely generated.
Let us show that \,$\mathrm{Fix} (Z^{(k)}(G)) \neq \emptyset$\, by induction on
\,$k$\,;
it is obvious for \,$k=0$.
Suppose \,$\mathrm{Fix} (Z^{(k)}(G)) \neq \emptyset$.
Since \,$G$\, is a finitely generated nilpotent group, it is polycyclic
(cf. \cite[Theorem 17.2.2]{Karga}) and as a consequence every subgroup
of \,$G$\, is finitely generated. Let \,$\{g_{1},\hdots,g_{n}\}$\, be a generator set
of \,$Z^{(k+1)}(G)$. We denote \,$H^{0}= Z^{(k)}(G)$\, and
\,$H^{j} = \langle Z^{(k)}(G), g_{1}, \hdots, g_{n} \rangle$ for any $0 < j \leq n$.
Notice that \,$H^{j}$\, is a normal subgroup of \,$H^{j+1}=\langle H^{j}, g_{j+1} \rangle$\,
for any \,$0 \leq j <n$.

Let us show that \,$\mathrm{Fix}(H^{j})$\, is a non-empty set for any  \,$j \geq 0$\,
by induction on \,$j$. The result is clear for  \,$j=0$.
Since \,$H^{j}$\, is normal in \,$H^{j+1}$, the set \,$\mathrm{Fix}(H^{j})$\, is \,$g_{j+1}$-invariant.
Let \,$\mu$\, be a \,$g_{j+1}$-invariant ergodic measure with
\,$\mathrm{supp}(\mu) \subset \mathrm{Fix}(H^{j})$.
The diffeomorphism \,$g_{j+1}$\, is irrotational; hence
we deduce \,$\mathrm{Fix}(H^{j+1}) \neq \emptyset$\,
by Theorem \ref{teo:main4}. We obtain
$\mathrm{Fix}(H^{j}) \neq \emptyset$ for any $0 \leq j \leq n$ and in particular
$\mathrm{Fix}(Z^{(k+1)}(G)) \neq \emptyset$.
As a consequence \,$\mathrm{Fix}(Z^{(k)}(G))$\, is a non-empty set for any \,$k \in {\mathbb Z}_{\geq 0}$.
Since \,$G$\, is nilpotent, we deduce \,$\mathrm{Fix}(G) \neq \emptyset$.

\vskip10pt
\subsection*{Proof of Theorem \ref{teo:main}}

Let us prove the result for
\,$G \subset \mathrm{Diff}^{1}({\mathbb T}^{2})$. The remaining cases
are analogous. The group \,$G'''$\, is irrotational by
Theorem \ref{teo:main3}.
Thus \,$\mathrm{Fix}(G''')$\, is non-empty
by Theorem \ref{teo:main2}.

\vskip30pt
\section{
Replacing \,$\mathbb{T}^{2}$\, with the  Klein bottle, the compact annulus or
the M\"obius strip
}

In this  section we remark that almost all the theorems in this article have  versions for
 the Klein bottle \,$\mathbb{K}^{2}$, the compact annulus
\,$\mathbb{S}^{1}\times [0\,,1]$\, and
the compact M\"obius strip.

First, let us remark that an isotopic to the identity \,$\mathbb{S}^{1}$-homeomorphism   has a trivial rotation number if and only if it has a fixed point. In this context
the following result can be interpreted as a version of Theorems
\ref{teo:main4} and \ref{teo:main2}.
Notice that solvable (and in particular nilpotent) groups are amenable.

\vskip10pt
\begin{pro}\label{T1:1:circle}
Let \,$G$\, be an amenable subgroup of
\,$\mathrm{Homeo}_{+}({\mathbb S}^{1})$\,
such that every element of \,$G$\, has a non-empty fixed-point set.
Then \,$G$\, has a global fixed point.
\end{pro}

\begin{proof}
Since \,$G$\, is amenable, there exists a $G$-invariant Borel probability measure \,$\mu$.
Given \,$\phi \in G$, the dynamics of \,$\phi$\, in every connected component of
\,$\mathbb{S}^{1} \setminus \mathrm{Fix}(\phi)$\, is conjugated to a translation.
Thus no point of \,$\mathbb{S}^{1} \setminus \mathrm{Fix}(\phi)$\,  can be in the support of a \,$\phi$-invariant
Borel probability measure. As a consequence
we obtain \,$\mathrm{supp}(\mu) \subset \mathrm{Fix}(G)$\,  and hence \,$\mathrm{Fix}(G)$\, is
non-empty.
\end{proof}


The versions of Theorems \ref{teo:main3} and \ref{teo:main}  for the circle
are very elementary.
 In this context  given an amenable
 subgroup \,$G$\,  of \,$\mathrm{Homeo}(\mathbb{S}^{1})$\, then\,:
\begin{itemize}
\item[$(i)$]
$G'$\, is irrotational when \,$G \subset \mathrm{Homeo}_{+}(\mathbb{S}^{1})$\,;

\item[$(ii)$]
$G''$\, is irrotational when \,$G \subset \mathrm{Homeo}(\mathbb{S}^{1})$.
\end{itemize}
Indeed the rotation number for elements of $G$ coincides with the rotation number
with respect to a $G$-invariant measure. Thus the map sending  each element of a subgroup $G$
of $\mathrm{Homeo}_{+}(\mathbb{S}^{1})$ to its
Poincar\'{e}'s rotation number is a morphism of groups by an analogue of Lemma \ref{lem:mor} for the circle.
Thus every element of $G'$ has vanishing rotation number.
Consequently, in cases $(i)$ and $(ii)$ above, there exist
global fixed points by Proposition \ref{T1:1:circle}.

\vglue10pt

In what follows let \,$\tau:\mathbb{T}^{2}\rightarrow\mathbb{K}^{2}$\, be the \,$2$-fold
orientation covering map of the Klein bottle by the \,$2$-torus and let us denote by
\,$\sigma$\, the non-trivial \,$\tau$-lift of the identity map.
Consider the covering transformations
\[ T_{1}(x,y) = (x+1/2, -y) \ \ \mathrm{and} \ \  T_{2}(x,y)= (x,y+1) \]
generating the fundamental group of ${\mathbb K}^{2} = {\mathbb R}^{2} / \langle T_1, T_2 \rangle$.
Moreover we have ${\mathbb T}^{2} = {\mathbb R}^{2} / \langle T_{1}^{2}, T_2 \rangle$.

Given a subgroup \,$G \subset \mathrm{Homeo}(\mathbb{K}^{2})$\, let us denote by
\,$G_{\tau} \subset \mathrm{Homeo}(\mathbb{T}^{2})$\, the subgroup of all
\,$\tau$-lifts of elements of \,$G$. As remarked in \cite[Section 4]{rib02}, given
\,$f\in G$\, its distinct  \,$\tau$-lifts are \,$\hat{f}$\, and
\,$\sigma \circ \hat{f}$\,.
Moreover \,$\hat{f}$\, and \,$\sigma$\,  commute and, consequently, \,$G_{\tau}$\, is
nilpotent when \,$G$\, has this property.

On the other hand for each \,$f\in \mathrm{Homeo}_{0}(\mathbb{K}^{2})$\,
there exists a unique \,$\tau$-lift \,$f_{\tau}$\, of \,$f$\, in
\,$\mathrm{Homeo}_{0}(\mathbb{T}^{2})$. It can be obtained by lifting via \,$\tau$\, any isotopy from the identity of \,$\mathbb{K}^{2}$\,  to \,$f$\, starting at the identity of
\,$\mathbb{T}^{2}$. Let us remark that \,$\sigma \circ f_{\tau}$\, cannot be obtained by this process since \,$\sigma$\, is orientation reversing. In this way we construct a natural
\,$\tau$-lift \,$G_{\tau,0}\subset \mathrm{Homeo}_{0}(\mathbb{T}^{2})$\, of any
subgroup  \,$G \subset\mathrm{Homeo}_{0}(\mathbb{K}^{2})$.

For a given  \,$f\in \mathrm{Homeo}(\mathbb{K}^{2})$\, and \,$\mu \in\mathcal{P}(f)$\, we denote by \,$\mu_{\tau}$\, the Borel probability measure on \,$\mathbb{T}^{2}$\,
obtained by normalizing the lift of \,$\mu$\, via the local homeomorphism \,$\tau$. In this case  the measure \,$\mu_{\tau}$\, is the unique probability measure invariant
by \,$f_{\tau}$\, and $\sigma$ (and then also by
\,$\sigma \circ f_{\tau}$) but we lost the ergodicity for the measure \,$\mu_{\tau}$
if \,$\mu$\, is ergodic.
Let us address this issue. There exists a natural map
\[
\begin{array}{ccccc}
 \tau_{*} & : & \mathcal{P}(f_{\tau}) & \to & \mathcal{P}(f) \\
& & \nu & \mapsto & \tau_{*} \nu
\end{array}
\]
where $(\tau_{*} \nu) (B) = \nu (\tau^{-1}(B))$ for any Borel subset $B$ of $\mathbb{T}^{2}$.
Let $\mu$ be a $f$-ergodic Borel probability measure.
The measure $\mu_{\tau}$ is the unique
Borel probability measure in $(\tau_{*})^{-1}(\mu)$ such that $\sigma_{*} \mu_{\tau} = \mu_{\tau}$.
The subset $(\tau_{*})^{-1}(\mu)$ of $ \mathcal{P}(f_{\tau})$ is non-empty (it contains $\mu_{\tau}$),
compact in the weak$^{*}$ topology and convex.
Hence $(\tau_{*})^{-1}(\mu)$ has a extreme point $\nu$ by the Bauer maximum principle
(cf. \cite[Corollary 25.10]{Choquet1969}).
We claim that $\nu$ is $f_{\tau}$-ergodic. Otherwise $\nu$ is not a extreme point among $f_{\tau}$-invariant
measures, i.e. there exist $f_{\tau}$-invariant Borel probability measures $\nu_1$, $\nu_2$
and $0 < t <1$ such that $\nu_1 \neq \nu_2$ and $\nu = t \nu_1 + (1-t) \nu_2$.
We deduce $\tau_{*} \nu = t \tau_{*} \nu_1 + (1-t) \tau_{*} \nu_2$.
Since $\tau_{*} \nu= \mu$ is $f$-ergodic, it follows that $\tau_{*} \nu$ is a extreme point of the set of
$f$-invariant Borel probability measures. We obtain
$\tau_{*} \nu =  \tau_{*} \nu_1  = \tau_{*} \nu_2= \mu$.
Now the equality $\nu = t \nu_1 + (1-t) \nu_2$ (with $0 < t < 1$ and $\nu_1 \neq \nu_2$)
contradicts that $\nu$ is a extreme point of $(\tau_{*})^{-1}(\mu)$.

Since $\nu$ is ergodic,
there exists a Borel set $A$ such that $\nu (A)=1$ and
\[ \nu = \lim_{n \to \infty}  \frac{\sum_{j=0}^{n-1} \delta_{f_{\tau}^{j}(p)}}{n} \]
in the weak* topology for any $p \in  A$, where $\delta_{q}$ is the Dirac delta at $q$. Moreover,
the equality $f_{\tau} \circ \sigma = \sigma \circ f_{\tau}$ implies
\[ \sigma_{*} \nu = \lim_{n \to \infty}  \frac{\sum_{j=0}^{n-1} \delta_{f_{\tau}^{j}(p)}}{n} \]
 for any $p \in  \sigma (A)$. By construction we have $\mu (\tau(A)) =1$ and
\[  \lim_{n \to \infty}  \frac{\sum_{j=0}^{n-1} \delta_{f_{\tau}^{j}(p)}}{n} \in \{ \nu, \sigma_{*} \nu  \} \]
for any $p \in \tau^{-1} (\tau (A))$.
Since $\nu'  ( \tau^{-1} (\tau (A)) )= 1$ for any probability measure in $(\tau_{*})^{-1}(\mu)$,
we deduce that $\nu$ and $\sigma_{*} \nu$ are the unique ergodic $f_{\tau}$-invariant probability
measures by Birkhoff's ergodic theorem. The measure $\frac{\nu + \sigma_{*} \nu}{2}$
belongs to $(\tau_{*})^{-1}(\mu)$ and is $\sigma$-invariant and hence it is equal to $\mu_{\tau}$.
Notice that $\mu_{\tau}$ is ergodic if and only if $\sigma_{*} \nu = \nu$ and then
$(\tau_{*})^{-1}(\mu) = \{ \mu_{\tau} \}$.

Let $\tilde{f}_{\tau}$ be an identity lift  of $f_{\tau}$ to ${\mathbb R}^{2}$. We denote
 $ \rho_{\nu}(\tilde{f}_{\tau})  = (a,b)$. We have

\[ \lim_{n \to \infty} \frac{\tilde{f}_{\tau}^{n}(x,y) - (x,y)}{n} = \rho_{\nu}(\tilde{f}_{\tau}) \ \ \mathrm{and} \ \
 \lim_{n \to \infty} \frac{\tilde{f}_{\tau}^{n}(T_1(x,y)) - T_1(x,y)}{n} = (a,-b) \]
for any $(x,y) \in \pi^{-1}(p)$ and $\nu$-a.e. $p \in {\mathbb T}^{2}$
by Birkhoff's ergodic theorem and simple calculations.
We deduce $\rho_{\sigma_{*} \nu}(\tilde{f}_{\tau})  = (a,-b)$ and then
$\rho_{\mu_{\tau}}(\tilde{f}_{\tau})  = (a,0)$.
Thus the vector $(a, |b|)$ does not depend on the choice of the ergodic measure in $(\tau_{*})^{-1}(\mu)$
and the following concept is well-defined.

\begin{defi}
Let \,$\mu$\, be an ergodic $f$-invariant measure and \,$\tilde{f}_{\tau}$\,
be an identity lift of
\,$f_{\tau}$\, to \,${\mathbb R}^{2}$. We define  \,$\overline{\rho}_{\mu}(\tilde{f}_{\tau})$\, as the vector
\,$(a,|b|)$\, where \,$(a,b) = \rho_{\nu} (\tilde{f}_{\tau})$\, for some ergodic \,$f_{\tau}$-invariant measure \,$\nu$\, with
\,$\tau_{*} \nu = \mu$.
\end{defi}

\begin{rem}
Notice that any other identity lift \,$\hat{f}_{\tau}$\, of \,$f_{\tau}$ to \,${\mathbb R}^{2}$\, is of the form
\,$T \circ \tilde{f}_{\tau}$\, where \,$T$\, is a covering transformation of \,${\mathbb T}^{2}$\, that commutes with
every map in \,$\langle T_1, T_2 \rangle$. As a consequence \,$T$\, is of the form
\,$T(x,y)= (x+m,y)$\, for some \,$m \in {\mathbb Z}$. We obtain
\,$\overline{\rho}_{\mu}(\hat{f}_{\tau})=(a+m, |b|)$.
\end{rem}

\begin{defi}
Let \,$\mu$\, be an ergodic $f$-invariant measure. We define  \,$\overline{\rho}_{\mu}({f}_{\tau})$\, as the class of the
vector \,$(a,|b|)$\, in \,${\mathbb R}/{\mathbb Z} \times {\mathbb R}$\,
where \,$(a,b) = \overline{\rho}_{\mu} (\tilde{f}_{\tau})$\, for some identity lift \,$\tilde{f}_{\tau}$\, to
\,${\mathbb R}^{2}$.
\end{defi}

The following lemma is a trivial consequence of the previous discussion.

\begin{lem}
\label{lem:rotveck}
Let \,$f\in \mathrm{Homeo}_{0}(\mathbb{K}^{2})$.
Consider an ergodic $f$-invariant Borel probability measure \,$\mu$.
Then the following conditions are equivalent:
\begin{itemize}
\item
$\overline{\rho}_{\mu}({f}_{\tau}) \in \{0 \} \times {\mathbb Z} \subset
{\mathbb R}/{\mathbb Z} \times {\mathbb R}$ \,\rm;

\item
${\rho}_{\nu}(f_{\tau}) = 0 \in \mathbb{R}^{2}/\mathbb{Z}^{2} $\,
for some ergodic $f_{\tau}$-invariant probability measure
\,$\nu \in (\tau^{*})^{-1}(\mu)$ \,\rm;

\item
${\rho}_{\nu}(f_{\tau}) = 0 \in \mathbb{R}^{2}/\mathbb{Z}^{2} $\,
for any ergodic $f_{\tau}$-invariant probability measure
\,$\nu \in (\tau^{*})^{-1}(\mu)$.
\end{itemize}
\end{lem}

Now we can introduce the analogue of Theorem \ref{teo:main4} for the Klein bottle.

\begin{cor}
\label{cor:main4}
Let \,$G= \langle H,\phi \rangle$\, be a nilpotent subgroup of
\,$\mathrm{Diff}_{0}^{1}({\mathbb K}^{2})$\,
where \,$H$\, is a normal subgroup of \,$G$. Suppose  there exists a $\phi$-invariant
ergodic probability measure \,$\mu$\, such that the support of \,$\mu$\,
is contained in
\,$\tau (\mathrm{Fix}(H_{\tau,0}))$\, and \,$ \overline{\rho}_{\mu}(\phi_{\tau}) \in \{0 \} \times {\mathbb Z}$.
Then \,$G$\, has a global fixed point.
\end{cor}

\begin{proof}
Let \,$\nu$\, be an  ergodic $\phi_{\tau}$-invariant Borel probability measure
\,$\nu \in (\tau^{*})^{-1}(\mu)$.
Given a point \,$p \in {\mathbb K}^{2}$\, in  \,$\mathrm{supp} (\mu)$, there exists
\,$p_{\tau} \in {\mathbb T}^{2} \cap \tau^{-1}(p) \cap \mathrm{Fix}(H_{\tau,0})$\, by hypothesis.
Since \,$\sigma$\, commutes with \,$g_{\tau}$\, for any \,$g \in H$,
we obtain \,$\mathrm{supp} (\nu) \subset \tau^{-1}(\mathrm{supp} (\mu)) \subset \mathrm{Fix}(H_{\tau,0})$.
Since \,$H_{\tau,0}$\, is normal in \,$\langle H_{\tau,0}, \phi_{\tau} \rangle$\, and
\,${\rho}_{\nu}(\phi_{\tau}) = 0 \in \mathbb{R}^{2}/\mathbb{Z}^{2} $ (Lemma \ref{lem:rotveck}), it follows that
\,$\langle H_{\tau,0}, \phi_{\tau} \rangle$\, has a global fixed point by Theorem \ref{teo:main4}.
Thus \,$\langle H, \phi \rangle$\, has a global fixed point in \,${\mathbb K}^{2}$.
\end{proof}

We can refine the previous result by replacing the condition
\,$\tau^{-1}(\mathrm{supp}(\mu)) \subset \mathrm{Fix}(H_{\tau,0})$\, with
\,$\mathrm{supp}(\mu) \subset \mathrm{Fix}(H)$.

\begin{pro}
\label{pro:main4}
Let \,$G= \langle H,\phi \rangle$\, be a nilpotent subgroup of
\,$\mathrm{Diff}_{0}^{1}({\mathbb K}^{2})$\,
where \,$H$\, is a normal subgroup of \,$G$. Suppose  there exists a $\phi$-invariant
ergodic probability measure \,$\mu$\, such that \,$\mathrm{supp}(\mu) \subset \mathrm{Fix}(H)$\,
and \,$ \overline{\rho}_{\mu}(\phi_{\tau}) \in \{0 \} \times {\mathbb Z}$.
Then \,$G$\, has a finite orbit with \,$2^{m}$\, elements where \,$m$\, is at most
the \,{\it nilpotency class}\, of \,$G$.
\end{pro}

\begin{proof}
Let \,$\nu$\, be an  ergodic \,$\phi_{\tau}$-invariant Borel probability measure
\,$\nu$\, with \,$\tau_{*} \nu = \mu$.
We denote \,$\psi_{j} = \phi_{\tau}^{2^{j}}$ for $j \geq 0$.
We define \,$\nu_0= \nu$.
Given \,$j \geq 0$\, we consider an ergodic \,$\phi_{j+1} = \phi_{j}^{2}$-invariant Borel probability measure
\,$\nu_{j+1}$\, such that
\[ \nu_{j} = \frac{\nu_{j+1} + \phi_{j}^{*} \nu_{j+1}}{2}. \]
There exists a lift \,$\tilde{\psi}_{0}$\, of \,$\psi_0$\, to \,${\mathbb R}^{2}$\, such that
\,$\rho_{\nu_0} (\tilde{\psi}_0) = (0,0)$\, by Lemma \ref{lem:rotveck}.
We denote \,$\tilde{\psi}_{j+1} = \tilde{\psi}_{j}^{2}$\, for any \,$j \geq 0$.
We have \,$\rho_{\nu_{j+1}} (\tilde{\psi}_{j+1}) =2 \rho_{\nu_j} (\tilde{\psi}_{j})$\, for any \,$j \geq 0$\, by
construction.

Let us show, by induction on \,$j$, that
\,$[{\psi}_{j}, \eta_{\tau}]$\, fixes  \,$\tau^{-1} (\mathrm{supp}(\nu_j))$\, pointwise for all
\,$j \geq 0$\, and \,$\eta \in Z^{(j+1)}(G) \cap H$.
The result is obvious for \,$j=0$\, since \,$[\psi_{0}, \eta_{\tau}] = Id$\, for any \,$\eta \in Z^{(1)}(G) \cap H$.
Suppose the result holds for \,$j \geq 0$. Given \,$\eta \in  Z^{(j+2)}(G) \cap H$\, we have
\,$\psi_{j} \circ  \eta_{\tau} \circ \psi_{j}^{-1} = h_{\tau} \circ \eta_{\tau}$\,
for some \,$h \in Z^{(j+1)}(G) \cap H$. We get
\[ \psi_{j}^{2} \circ  \eta_{\tau} \circ \psi_{j}^{-2} =
( \psi_{j} \circ h_{\tau} \circ \psi_{j}^{-1})  \circ h_{\tau} \circ  \eta_{\tau}. \]
Since \,$[\psi_{j}, h_{\tau}]$\, fixes \,$\tau^{-1} (\mathrm{supp}(\nu_j))$\, by induction hypothesis and
\,$\mathrm{supp}(\nu_{j+1}) \subset \mathrm{supp}(\nu_{j})$,
it follows that \,$\psi_{j} \circ h_{\tau} \circ \psi_{j}^{-1} = h_{\tau}' \circ h_{\tau}$\, for some \,$h' \in Z^{(j)}(G) \cap H$\,
such that \,$h_{\tau}'$\, fixes \,$\tau^{-1} (\mathrm{supp}(\nu_{j+1}))$\, pointwise.
Since \,$[\psi_{j+1}, \eta_{\tau}] = h_{\tau}' \circ h_{\tau}^{2}$, it follows that \,$[\psi_{j+1}, \eta_{\tau}]$\,
fixes  \,$\tau^{-1} (\mathrm{supp}(\nu_{j+1}))$\, pointwise for any \,$\eta \in Z^{(j+2)}(G) \cap H$.

Let  \,$n$\, be the nilpotency class of \,$G$\, and \,$r = 2^{n-1}$. We denote \,$\psi = \phi^{r}$.
Given \,$\eta \in H$, we define the subset \,$A_{\eta}$\, of \,$\tau(\mathrm{sup}(\nu_{n-1}))$\,
of points \,$p$\, such that \,$\tau^{-1}(p) \subset \mathrm{Fix}(\eta_{\tau})$.
The property \,$\mathrm{sup}(\nu_{n-1}) \subset \mathrm{Fix} [\psi_{\tau}, \eta_{\tau}]$\,
implies that \,$A_{\eta}$\, is \,$\psi$-invariant. Now consider a \,$\nu_{n-1}$\, generic point
\,$p \in {\mathbb T}^{2}$\, and its \,$\psi_{\tau}$-orbit ${\mathcal O}$.
We have that \,$\tau ({\mathcal O})$\, is contained either in \,$A_{\eta}$\, or its complement.
Since \,$\overline{\mathcal O} = \mathrm{sup}(\nu_{n-1})$\, and both sets
\,$A_{\eta}$\, and \,$\tau (\mathrm{sup}(\nu_{n-1})) \setminus A_{\eta}$\, are
open and closed in \,$\tau(\mathrm{sup}(\nu_{n-1}))$, we deduce that
\,$A_{\eta}  = \emptyset$\, or \,$A_{\eta} = \tau(\mathrm{sup}(\nu_{n-1}))$. We define
\[ H_{+} = \{\eta \in H : \mathrm{sup}(\nu_{n-1}) \subset \mathrm{Fix}(\eta_{\tau}) \}. \]
Since \,$H$\, is a normal subgroup of \,$G$,  it follows that \,$H_{+}$\, is also a normal
subgroup of \,$\langle H_{+}, \psi \rangle$\, of index at most \,$2$\, in \,$H$.
By construction we have \,$\mathrm{sup}(\nu_{n-1}) \subset \mathrm{Fix} (H_{+})_{\tau,0}$.
This condition together with \,$\rho_{\nu_{n-1}} (\tilde{\psi}_{n-1})=(0,0)$\, allows to apply Corollary
\ref{cor:main4} to obtain a global fixed point of \,$\langle (H_{+})_{\tau,0}, \psi_{\tau} \rangle$\, and then a global
fixed point \,$p$\, of \,$\langle H_{+}, \psi \rangle$.
It is clear that \,$|\langle H, \phi \rangle : \langle H, \psi \rangle|$\, is of the form \,$2^{a}$\, for some
\,$0 \leq a \leq n-1$\, since \,$\phi^{r} = \psi$ and $r=2^{n-1}$.
Since \,$H$\, is a normal subgroup of \,$G$, it follows that \,$|\langle H, \psi \rangle : \langle H_{+}, \psi \rangle|$\,
is either \,$1$\, or \,$2$. Thus the index \,$|\langle H, \phi \rangle : \langle H_{+}, \psi \rangle|$\, is of the form \,$2^{b}$\, with \,$0 \leq b \leq n$. Consider the subgroup \,$J$\, of \,$G$\, such that
\,$J= \{ \eta \in G : \eta(p)=p\}$. The group \,$J$\, contains  \,$\langle H_{+}, \psi \rangle$\, and thus
\,$|G:J|$\, is of the form \,$2^{m}$\, with \,$0 \leq m \leq n$. Since \,$|G:J|$\, is the cardinal of the $G$-orbit of \,$p$\, the
result is proved.
\end{proof}

Let us focus now in the generalization of Theorem  \ref{teo:main2} for the Klein bottle.

\begin{defi}
We say that \,$f\in \mathrm{Homeo}_{0}(\mathbb{K}^{2})$\, is \,{\it irrotational}\, when the
\,$\tau$-lift \,$f_{\tau}\in \mathrm{Homeo}_{0}(\mathbb{T}^{2})$\, is irrotational.
\end{defi}

\begin{rem}
An alternative definition of irrotational element of \,$\mathrm{Homeo}(\mathbb{K}^{2})$\, is the following:
an element \,$f$\, of \,$\mathrm{Homeo}(\mathbb{K}^{2})$\, is irrotational if its orientation-preserving lift
\,$f_{+}$\, to \,${\mathbb T}^{2}$\, is irrotational. A priori the definitions could be different since a non-isotopic
to the identity \,$f \in \mathrm{Homeo}({\mathbb K}^{2})$\, can have an isotopic to the identity lift \,$f_{+}$.
An example is the homeomorphism induced by \,$\tilde{f}(x,y) = (x,y+1/2)$\, in the Klein bottle.
Anyway we claim that both definitions are equivalent. It suffices to
show that if \,$f$\, satisfies the second definition then \,$f$\, belongs to \,$\mathrm{Homeo}_{0}(\mathbb{K}^{2})$.
Let \,$\tilde{f}$\, be the irrotational lift of \,$f_{+}$. The sequence of maps \,$\frac{\tilde{f}^{n}-Id}{n}$\, converges
uniformly to the constant map \,$(0,0)$\, when \,$n \to \infty$. Since \,$f_{+}$\, is isotopic to the identity, it follows that
\,$\tilde{f} \circ T = T \circ \tilde{f}$\, for any \,$T \in \langle T_{1}^{2}, T_{2} \rangle$.
Moreover since \,$f_{+} \circ \sigma = \sigma \circ f_{+}$, we obtain
\[ \tilde{f} \circ T_{1} \circ \tilde{f}^{-1} = (x+a, y+b) \circ T_1 \]
for some \,$(a,b) \in {\mathbb Z}^{2}$. We deduce
\[ \tilde{f}^{n} \circ T_{1} \circ \tilde{f}^{-n} = (x+na, y+nb) \circ T_1 \]
for any \,$n \in {\mathbb Z}$. This leads us to
\[ \frac{\tilde{f}^{n}   - Id}{n} \circ T_1 = \frac{ (x+na, y+nb) \circ T_1 \circ \tilde{f}^{n} - T_1}{n} =
(a,b) +  (x, -y) \circ \frac{\tilde{f}^{n}-Id}{n}. \]
The left hand side of the previous expression converges uniformly to \,$(0,0)$\, whereas the right hand side
converges uniformly to \,$(a,b)$\, when \,$n \to \infty$.
Hence we get \,$(a,b)=(0,0)$\, and thus \,$\tilde{f}$\, commutes with every
covering transformation in \,$\langle T_1, T_2 \rangle$. As a consequence \,$f$\, is isotopic to the identity map.
\end{rem}

We have the following immediate corollary of Theorem  \ref{teo:main2}.

\vskip10pt
\begin{cor}\label{teo:main2:corol:strip}
Let \,$G$\, be an irrotational nilpotent subgroup of \,$\mathrm{Diff}_{0}^{1}({\mathbb K}^{2})$.
Then \,$G$\, has a global fixed point.
\end{cor}

Before to announce the corollaries of Theorems \ref{teo:main3} and \ref{teo:main} let us recall
another property of the group \,$G_{\tau}$.

Given  a subgroup \,$G\subset\mathrm{Homeo}(\mathbb{K}^{2})$\, let us denote by
\,$G_{\tau,+}$\, the subgroup of all the elements of \,$G_{\tau}$\, that are orientation preserving. It is not difficult to see that \,$G_{\tau,+}$\, is a \,$\tau$-lift of
\,$G$.
Since \,$f_{\tau}$\, and \,$\sigma$\, commute for all \,$f\in G$\, then we can also conclude that
\,$[\,G_{\tau}\,,G_{\tau}\,]=[\,G_{\tau,+}\,,G_{\tau,+}\,]$\,.

\vskip10pt
\begin{cor}\label{teo:main3:corol:strip}
If \,$G$\, is a nilpotent subgroup of \,$\mathrm{Homeo}({\mathbb K}^{2})$ then$\,:$
\begin{itemize}
\item[$(i)$]
$G'$\, is irrotational when \,$G \subset \mathrm{Homeo}_{0}({\mathbb K}^{2})\,;$

\item[$(ii)$]
$G''$\,  is irrotational when  \,$G \subset \mathrm{Homeo}({\mathbb K}^{2})$\,.

\end{itemize}
\end{cor}

\vskip10pt
\begin{cor}\label{teo:main:corol:strip}
Let \,$G$\, be a nilpotent subgroup of \,$\mathrm{Diff}^{1}({\mathbb K}^{2})$. The following subgroups of \,$G$\, have a global fixed point$\,:$
\begin{itemize}
\item[$(i)$]
$G'$\, when \,$G \subset \mathrm{Diff}^{1}_{0}({\mathbb K}^{2})\,;$

\item[$(ii)$]
$G''$\, when \,$G \subset \mathrm{Diff}^{1}({\mathbb K}^{2})$\,.

\end{itemize}
\end{cor}


For the compact annulus we have the following version for Theorem \ref{teo:main4}.

\vskip10pt
\begin{cor}\label{teo:main4:corol:01}

Let \,$\mathcal{N}=\langle \mathcal{H},\psi \rangle$\, be a nilpotent subgroup of
\,$\mathrm{Diff}^{1}_{0}(\mathbb{S}^{1}\!\times[0\,,1])$\,
where \,$\mathcal{H}$\, is a normal subgroup of  \,$\mathcal{N}$.
Suppose there exists a \,$\psi$-invariant ergodic probability measure \,$\mu$\, such that
the support of \,$\mu$\, is contained in \,$\mathrm{Fix}(\mathcal{H})$\, and
\,$\rho_{\mu}(\psi)$\, is the zero element in \,$\mathbb{R}/\mathbb{Z}$. Then
\,$\mathcal{N}$\, has a global fixed point.

\end{cor}
\vskip5pt

\begin{proof}

To detect a global fixed point for \,$\mathcal{N}$\, we consider the double
\,$\mathbb{T}^{2}$\, of the annulus and the double of each element of \,$\mathcal{N}$.
This construction give us a nilpotent subgroup \,$G=\langle H,\phi \rangle$\, of
\,$\mathrm{Homeo}_{0}(\mathbb{T}^{2})$\, where \,$H$\, is a normal subgroup of \,$G$\, and \,$\phi$\, is the double of \,$\psi$. Moreover, let us fix one of the copies of the annulus in the double \,$\mathbb{T}^{2}$\, and let us consider the measure
\,$\mu$\, on this copy as a measure on the double \,$\mathbb{T}^{2}$. We have that \,$\mu$\, is a
\,$\phi$-invariant ergodic probability measure on \,$\mathbb{T}^{2}$\, whose support is contained
in \,$\mathrm{Fix}(H)$\, and \,$\rho_{\mu}(\phi)$\, is the null element of
\,$\mathbb{R}^{2}/\mathbb{Z}^{2}$.

Let us recall that in the above construction the group \,$G$\, is a subgroup of
\,$\mathrm{Homeo}_{0}(\mathbb{T}^{2})$\,.
At this point to apply  Theorem \ref{teo:main4} we can smooth the action by using
a result of K. Parkhe in
\cite{Kiran}. He proves (cf. Theorem 5 and Remark 6 in section 2) that there exists a
homeomorphism of the annulus, supported in a small neighborhood of its boundary such that: the conjugate
of each element of \,$\mathcal{N}$\, by this homeomorphism gives a \,$C^{1}$-diffeomorphism that can be glued with
itself  to obtain a \,$C^{1}$-\,diffeomorphism in the double \,$\mathbb{T}^{2}$ of the annulus.

Of course the \,$C^{1}$-\,diffeomorphism corresponding to the map \,$\phi$\, and the measure corresponding to the measure \,$\mu$\, via the conjugation
 stay with the same properties  as above. Now, we can apply Theorem \ref{teo:main4} to conclude the existence of a global fixed point
in the annulus.
\end{proof}

Notice that if  \,$\mathrm{supp}(\mu)$\, is contained in the annulus boundary \,${\mathcal C}$\, then
\,$\mathcal{N}$\, has a global fixed point in \,$\mathcal{C}$\, even when \,$\mathcal{N}$\, is
contained in \,$\mathrm{Homeo}_{0}(\mathbb{S}^{1}\!\times [0\,,1])$\, by Proposition
\ref{T1:1:circle}.

\vglue10pt

Following the strategy of passing to the double of the annulus as in the proof of Corollary \ref{teo:main4:corol:01} we can easily adapt the statements of
Theorems \ref{teo:main2} , \ref{teo:main3} and \ref{teo:main} for the case \,$\mathbb{S}^{1}\! \times [0\,,1]$\,.
Since \,$G' \subset \mathrm{Homeo}_{0} (\mathbb{S}^{1}\! \times [0\,,1])$,  the second derived
group \,$G''$\, has always a global fixed point.
In a similar way we can give versions of Corollaries \ref{cor:main4}, \ref{teo:main2:corol:strip},
\ref{teo:main3:corol:strip} and \ref{teo:main:corol:strip} and Proposition \ref{pro:main4}
for the compact M\"obius strip.

\vglue40pt

\def\cprime{$'$}
\providecommand{\bysame}{\leavevmode\hbox to3em{\hrulefill}\thinspace}
\providecommand{\MR}{\relax\ifhmode\unskip\space\fi MR }
\providecommand{\MRhref}[2]{%
  \href{http://www.ams.org/mathscinet-getitem?mr=#1}{#2}
}
\providecommand{\href}[2]{#2}

\end{document}